\documentclass[reqno,english]{amsart}
\usepackage{amsfonts,amsmath,latexsym,verbatim,amscd,mathrsfs,color,array}

\usepackage{amsmath,amssymb,amsthm,amsfonts,graphicx,color}
\usepackage{amssymb}
\usepackage{pdfsync}
\usepackage{epstopdf}
\usepackage[colorlinks=true]{hyperref}

\makeatletter
\def\@currentlabel{2.1}\label{e:dispaa}
\def\@currentlabel{2.21}\label{e:dispau}
\def\@currentlabel{2.22}\label{e:dispav}
\def\@currentlabel{2.23}\label{e:dispaw}
\def\@currentlabel{2.24}\label{e:dispax}
\def\theequation{\thesection.\@arabic\c@equation}
\makeatother

\renewcommand{\theequation}{\thesection.\arabic{equation}}
\newtheorem{lemma}{Lemma}[section]

\newtheorem{proposition}{Proposition}[section]
\newtheorem{corollary}{Corollary}[section]
\newtheorem{remark}{Remark}[section]
\newcommand{\bremark}{\begin{remark} \em}
\newcommand{\eremark}{\end{remark} }

\newtheorem{theorem}{Theorem}[section]

\newcommand{\R}{{\mathbb R}}

\newcommand{\BE}{\begin{equation}}
\newcommand{\BEN}{\begin{equation*}}
\newcommand{\EE}{\end{equation}}
\newcommand{\EEN}{\end{equation*}}
\newcommand{\BL}{\begin{lemma}}
\newcommand{\EL}{\end{lemma}}
\newcommand{\BT}{\begin{theorem}}
\newcommand{\ET}{\end{theorem}}
\newcommand{\BP}{\begin{proposition}}
\newcommand{\EP}{\end{proposition}}
\newcommand{\BC}{\begin{corollary}}
\newcommand{\EC}{\end{corollary}}

\begin{document}

\title[Superfluids Passing an Obstacle]{Superfluids  Passing an Obstacle and Vortex Nucleation}

\author{Fanghua Lin}
\address{ Courant Institute of Mathematical Sciences, 251 Mercer Street, New York, NY. 10012}
\email{linf@cims.nyu.edu}

\author{Juncheng Wei}
\address{\noindent J. Wei - Department of Mathematics, University of British Columbia, Vancouver BC V6T 1Z2, Canada.}
\email{wei@math.cuhk.edu.hk}


\date{\today}\maketitle

\begin{abstract}
We consider a superfluid described by the Gross-Pitaevskii equation passing an obstacle
\[\epsilon^2 \Delta u+ u(1-|u|^2)=0 \ \mbox{in} \ \R^d \backslash \Omega, \ \  \frac{\partial u}{\partial \nu}=0 \ \mbox{on}\ \partial \Omega
\]
where $ \Omega$ is a smooth bounded domain in $ \R^d$ ($d\geq 2$), which is referred as the
obstacle and $ \epsilon>0$ is sufficiently small. We first construct a vortex free solution  of the form $ u= \rho_\epsilon (x)
e^{i \frac{\Phi_\epsilon}{\epsilon}}$ with $ \rho_\epsilon (x) \to 1-|\nabla \Phi^\delta(x)|^2,
\Phi_\epsilon (x) \to \Phi^\delta (x) $ where $\Phi^\delta (x)$ is the unique solution for the subsonic irrotational flow equation
\[ \nabla ( (1-|\nabla \Phi|^2)\nabla \Phi )=0 \ \mbox{in} \ \R^d \backslash \Omega, \ \frac{\partial
\Phi}{\partial \nu} =0 \ \mbox{on} \ \partial \Omega, \ \nabla \Phi (x) \to \delta \vec{e}_d \ \mbox{as} \ |x| \to +\infty \]
and $|\delta | <\delta_{*}$ (the sound speed).

In dimension $d=2$, on the background of this vortex free solution we also construct solutions with
single vortex close to the maximum or minimum points of the function $|\nabla \Phi^\delta
(x)|^2$ (which are on the boundary of the obstacle). The latter verifies the vortex nucleation phenomena
(for the steady states) in superfluids described by the Gross-Pitaevskii
equations. Moreover, after
some proper scalings, the limits of these vortex solutions are traveling wave solution  of the
Gross-Pitaevskii equation. These results also show rigorously the conclusions drawn from the
numerical computations in \cite{huepe1, huepe2}.

Extensions to Dirichlet boundary conditions, which may be more consistent with the situation in the
physical experiments and numerical simulations (see \cite{ADP} and references therein) for the trapped
Bose-Einstein condensates, are also discussed.

\end{abstract}

\setcounter{equation}{0}
\section{Introduction}

This paper addresses the vortex nucleation and vortex shedding phenomena when a superfluid passes an
obstacle. Of particular concern is the existence of associated steady
state solutions of the following Gross-Pitaevskii equation passing an
obstacle
\begin{equation}
\label{1.1}
\left\{\begin{array}{l}
\epsilon^2 \Delta u+ u(1-|u|^2)=0 \ \mbox{in} \ \R^d  \backslash \Omega, \\
 \ \frac{\partial u}{\partial \nu}=0 \ \mbox{on}\ \partial \Omega
\end{array}
\right.
\end{equation}
where $\epsilon>0$ is a sufficiently small constant, $\Omega$ is a smooth and bounded domain in $\R^d, \ d\geq 2$ and  $\nu$ denotes the unit outer normal.

\medskip

Equation (\ref{1.1}), defined in the whole space, or in a bounded domain
or in an exterior of a bounded domain,
arises in many physical problems. The solutions are often used to describe stationary flows for
superfluids, \cite{ADP, fpr,  jp, jpr, LL, lz-arma, pnb, R}, the traped Bose-Einstein condensates and
phenomena in nonlinear optics, \cite{GR, huepe1, huepe2, JR1, JR2}.

\medskip

The purpose of this paper is to construct {\em two} types of solutions to (\ref{1.1}). This will be
achieved  by perturbation of {\em two} basic solution profiles. The first basic
profile, which is a relatively clear one,  is
obtained through the so-called  Madelung transformation,
\begin{equation}
u=\rho e^{i \frac{\Phi}{\epsilon}}.
\end{equation}
In other words, one is interested in solutions in the semiclassical regime (i.e. $\epsilon$ being sufficiently small), see \cite{lz-arma} and
references therein for discussions in evolutionary cases.
Equation (\ref{1.1}) then  becomes
\begin{equation}
\label{1.2}
\left\{\begin{array}{l}
\epsilon^2 \Delta \rho+ \rho(1-\rho^2 -|\nabla \Phi |^2)=0 \ \mbox{in} \ \R^d \backslash \Omega, \\
\nabla (\rho^2 \nabla \Phi)= 0\ \mbox{in} \ \  \R^d \backslash \Omega,
\\
 \ \frac{\partial \rho}{\partial \nu}= \frac{\partial \Phi}{\partial \nu}=0 \  \ \mbox{on}\ \partial \Omega.
\end{array}
\right.
\end{equation}
Formally if we set $\epsilon=0$ and neglect for the moment the boundary condition for $\rho$, (\ref{1.2}) becomes
the standard irrotational flow equation passing through an obstacle
\begin{equation}
\label{1.3}
\left\{\begin{array}{l}
\nabla (\rho^2 \nabla \Phi)= 0\ \mbox{in} \ \R^d \backslash \Omega,
\\
\rho^2=1-|\nabla \Phi |^2, \\
 \frac{\partial \Phi}{\partial \nu}=0\   \mbox{on}\ \partial \Omega,\\
 \nabla \Phi (x) \to (0, 0,..., \delta):=\delta\vec{e}_d \ \mbox{as} \ |x| \to +\infty.
\end{array}
\right.
\end{equation}

There are classical works by L. Bers  \cite{bers1, bers2} in the two
dimensional case, R.Finn-Gilbarg \cite{fg} and G. Dong \cite{do} in higher
dimensional cases. We summarize the basic results in the following theorem.

\begin{theorem}
\label{t0}

\noindent
(i) There exists a $\delta_{*} \in (0,1)$ such that for $|\delta | < \delta_{*}$, there
exists a unique classical solution $\Phi= \Phi^\delta$ (steady state solution of (\ref{1.3})). For $|\delta |>\delta_{*}$, there are no classical solutions (the so-called shocks develop). (ii) The solution $\Phi^\delta $ has the property that  for  $|\delta |<\delta_{*}$
\begin{equation}
\label{phi1}
\max_{x \in \R^d \backslash \Omega}  |\nabla \Phi^\delta (x)|<\frac{1}{3}.
 \end{equation}

\end{theorem}

In the above theorem, $\delta_{*}$ is called the {\em sound speed } of this problem. Thus for classical
fluids passing an obstacle, the situation is relatively clear. For a fluid with a speed less than
the sound speed, there is a smooth stationary flow. When the speed of fluids goes beyond a critical
(sound) speed, there are no smooth stationary flows (shock develops). From the
semiclassical limit formalism
(see \cite{lz-arma} and the references therein), one would expect a similar conclusion may be also true for
superfluids passing an obstacle. Since superfluids are frictionless, there are no notions of
"shock" in this case. Instead, there would be nucleation of vortices and the latter would introduce a
dissipation mechanism that eventually destroy the superfluidity, see for examples \cite{LL, fpr, jp,
jpr, pnb, R, JR1, JR2}. However, there is no rigorous mathematical proof. Throughout this paper, we
always assume that

\begin{equation}
\label{1.4}
|\delta |<\delta_{*}.
\end{equation}

For superfluids passing an obstacle described by the equation (\ref{1.1}), our first result concerns
vortex free solutions in the "subsonic" case. It can be considered as perturbation from the solutions
of (\ref{1.3}). See also \cite{lz-arma} for rigorous verification in the evolutionary case.

\begin{theorem}
\label{t1}
Let $|\delta |<\delta_{*}$ be fixed and $d=2$ or $3$. Then there exists $\epsilon_0>0$ (which may depend on $\delta_{*}-|\delta|$) such that for $0<\epsilon <\epsilon_0$  problem (\ref{1.1}) has a smooth solution of the form $ u_\epsilon (x)= \rho_\epsilon (x) e^{i\frac{\Phi_\epsilon}{\epsilon}}$ such that as $\epsilon \to 0$,
$\nabla \Phi_\epsilon (x) \to \nabla \Phi^\delta (x), \rho_\epsilon (x) \to \rho^\delta (x):=1-|\nabla \Phi^\delta (x)|^2$, uniformly in $ \R^d \backslash \Omega$,  where $\Phi^\delta$ is the solution given by Theorem \ref{t0}.
\end{theorem}

   Though the conclusion of the above theorem may be expected (and it has been often
assumed in many
physics literature), it lacks of a rigorous mathematical proof. In fact, even at a low
superfluid speed, we shall see that certain boundary layers near the obstacle may develop
and this causes difficulties for analysis. See (\ref{s2.3}) below.

\medskip

Let $ u_\epsilon =\rho_\epsilon e^{i \frac{\Phi_\epsilon}{\epsilon}}$ be the solution constructed in Theorem \ref{t1}. It turns out that when $d=2$ on the top of this solution a second solution exists. Interestingly the limiting profile of this second solution is the {\em traveling wave} solutions of Gross-Pitaevaskii equation. More precisely, set
\begin{equation}
 u= u_\epsilon \ v= \rho_\epsilon e^{i \frac{\Phi_\epsilon}{\epsilon}} \ v.
 \end{equation}
 Then $v$ satisfies (coupled with homogeneous Neumann boundary condition)
 \begin{equation}
 \epsilon^2 \Delta v + 2\epsilon^2 \nabla \rho_\epsilon \nabla v+ 2i \epsilon \nabla \Phi_\epsilon \nabla v + v \rho_\epsilon^2 (1- |v|^2)=0.
 \end{equation}

 Let $ x_0\in \partial \Omega$ and perform a rescaling as follows: $ x= x_0+\epsilon y$. Formally letting $\epsilon \to 0$, (and after proper scaling),  we obtain (assuming that $ \nabla \Phi^\delta (x_0)= |\nabla \Phi^\delta (x_0)| \vec{e}_2$) the following traveling wave equation
 \begin{equation}
\label{limit2.1}
\Delta U+  i c \frac{\partial U}{\partial y_2} +  U (1-|U|^2)=0\   \ \ \ \mbox{in} \ \R^2
\end{equation}
coupled with the following boundary condition
\begin{equation}
\label{limit2.2}
\frac{\partial U}{\partial y_1} (0, y_2)=0.
\end{equation}
We refer to Section 3.1 for more detailed derivations.

Here
\begin{equation}
\label{phi2}
 c = \frac{ 2 |\nabla \Phi^\delta (x_0)|}{ \sqrt{1-|\nabla \Phi^\delta (x_0)|^2}}.
 \end{equation}
 A simple computation shows that the subsonic condition (\ref{phi1}) is equivalent to the following speed condition
 \begin{equation}
  0<c <\sqrt{2}.
  \end{equation}

The traveling wave problem (\ref{limit2.1})-(\ref{limit2.2}) for Gross-Pitaevskii equation
 has been under study in many papers \cite{BS, bgs, grav1, grav2, grav3, grav4,
 lw-cpam}. It has been proved that for $c \geq \sqrt{2}$ there are no
traveling waves (\cite{grav4}). When $c <\sqrt{2}$, variational method
shows  that there are traveling wave solutions to (\ref{limit2.1}).
For $c$ small a perturbation argument can be used to show the
existence (\cite{lw-cpam}). The properties of solutions  constructed in \cite{lw-cpam}
would play an important role in the proof of our main theorem below.  The
asymptotic behavior and qualitative behavior of solutions are also
studied in many papers \cite{grav1, grav2, grav3, grav4}. What is
remarkable and fascinating is the fact that the critical speed for existence of
traveling wave solutions for the Gross-Pitaevskii equation on the entire
plane is directly related to the critical (sound) speed for stationary
flows of superfluids passing a smooth obstacle described by the Gross-Pitaevskii
equation through an explicit but nonlinear algebra relation (\ref{phi2}).

\medskip

Our second result shows the traveling wave solutions to (\ref{limit2.1}) persist for the superfluids
problem (\ref{1.1}), as long as $ |\delta |$ is suitably small (see Section 3 below).

\begin{theorem}
\label{t2}
Let $d=2$ and  $0<c<\sqrt{2}$. Let $U_{c}$ be a  solution of (\ref{limit2.1})-(\ref{limit2.2}) satisfying a nondegeneracy condition. (See Key Assumption (\ref{ker12}) below.)  Then   there exists $c_0>0$ and $\epsilon_0$  such that for $0<\epsilon <\epsilon_0, |c|<c_0$  problem (\ref{1.1}) has at least {\bf two smooth solutions}  of the form
\begin{equation}
u=(1+o(1)) u_\epsilon (x) U_{c } \left(\sqrt{1-|\Phi^\delta (x_0) |^2} \  \frac{x-x_\epsilon}{\epsilon} \right)
\end{equation}
where $u_\epsilon $ is the solution given by Theorem \ref{t1} and $U_{c}$ is the traveling wave solution of  problem (\ref{limit2.1})-(\ref{limit2.2}). Here $c$ is given by (\ref{phi2}),  $x_\epsilon \in \partial \Omega \to x_0$ where $  |\nabla \Phi^{\delta} |^2 (x_0)= \max_{x \in \partial \Omega} |\nabla \Phi^{\delta} |^2 (x)$ or $  |\nabla \Phi^{\delta} |^2 (x_0)= \min_{x \in \partial \Omega} |\nabla \Phi^{\delta} |^2 (x)$.
\end{theorem}

Theorem \ref{t2} shows not only the phenomena of vortex nucleations in stationary flows of
superfluids but also a somewhat surprising new phenomena that even before the "sonic" speed vortices
can nucleate near the boundary of the obstacle. It does rigorously justify
some seemly strange conclusions drawn from previous numerical studies in \cite{huepe1, huepe2}.
We believe that the second solution exists for all subsonic speed $ c < \sqrt{2}$.
But this remains to be an open problem and we wish to return to this issue later, see remarks in
Section 3. Theorem \ref{t2} also shows the situation near the superfluid "sonic" speed
may be much more complex than some formal studies done previously, \cite{fpr, jp,
jpr, R}, and the study of the latter situation would need a new set of tools and
ideas.

\medskip

As we mentioned earlier, the proofs of Theorems \ref{t1} and \ref{t2} are based on perturbations of {\em two} kinds of solutions.
In the proof of Theorem \ref{t1} the primary ansatz is the solution to (\ref{1.3}). The linearized operator is a system whose Fourier transforms are {\em anisotropic}. The additional difficulty is the existence of {\em boundary layers}. We use energy method and a priori estimates to prove Theorem \ref{t1}.   Theorem \ref{t2} is perturbed from a traveling wave solution to the Gross-Pitaevskii equation (\ref{limit2.1}) in the whole space.
Under a nondegenerate condition of solutions to (\ref{limit2.1}), which we verify for $ |c |<<1$, we prove Theorem \ref{t2} by finite dimensional Liapunov-Schmidt reduction method.

\medskip

Throughout this paper, we always assume that $d=2$ or $3$ (which are the physical dimensions). (The result of Theorem \ref{t1} is likely to hold for $d\geq 4$ as well.) The constant $C$ is a  positive generic constant independent of $\epsilon<\epsilon_0$ and $c<c_0$. Denote  $ B_\rho (y)= \{ x \ | \ |x-y|<\rho\}$. We also use the following notation
$$ <y>:= \sqrt{1+|y|^2}, \ \ \  <f, g>= Re (\int f \bar{g}). $$

\medskip

\noindent {\bf Acknowledgments:}  The research of the first author is
partially supported by the NSF grant DMS-1501000. The
research of the second author is partially supported by NSERC of Canada.

\setcounter{equation}{0}
\section{Proof of Theorem \ref{t1} }

In this section, we prove the existence of vortex free solution, e.g. Theorem \ref{t1}.

First we introduce two nonlinear operators
\begin{equation}
\label{S1S2}
 S_1 [\rho, \Phi]= \epsilon^2 \Delta \rho +\rho (1-\rho^2 -|\nabla \Phi |^2), \ \ S_2 [\rho, \Phi] =\nabla (\rho^2 \nabla \Phi).
\end{equation}
Then equation (\ref{1.2}) can be written in an operator form
\begin{equation}
\label{2.1}
S_1 [\rho, \Phi]= S_2 [\rho, \Phi]=0, \ \ ( \rho, \Phi) \in ( C_\nu^{2, \alpha} ( \R^d \backslash \Omega))^2
\end{equation}
where $ C^{2, \alpha}_\nu (\R^d \backslash \Omega)= C^{2, \alpha} (\R^d \backslash \Omega) \cap \{ \frac{\partial \rho}{\partial \nu}=0 \ \mbox{on} \ \partial \Omega \} $ and $\alpha \in (0,1)$ is the H\"{o}lder exponent.

As mentioned in the introduction, formally letting $\epsilon=0$ in equation (\ref{2.1}), we obtain the limiting problem (\ref{1.3}).  We first collect some additional  properties of solutions to (\ref{1.3}),  which will be needed in later sections.

 \begin{lemma}
 \label{t2.11}

Let $ \Phi^\delta$ be the solution to (\ref{1.3}) given  in Theorem \ref{t2.11}. Then as $|x| \to +\infty$,
 \begin{equation}
 \label{rhophidecay}
 \nabla \Phi^\delta (x) = \delta \vec{e}_d  + {\mathcal O} (\frac{1}{|x|}), \ \ \  1- |\nabla \Phi^\delta |^2 = 1- \delta^2 + {\mathcal O} (\frac{1}{|x|}).
 \end{equation}

 \end{lemma}

\begin{proof}
The asymptotic behavior  can be found in [Theorem III, \cite{bers1}] (in the case of $d=2$) and \cite{do} (in the case $d=3$).
\end{proof}

\subsection{Approximation solution and boundary layer}

For the first approximation function, we take
$$
 W_0= (\rho^\delta, \Phi^\delta )
$$
where $\Phi^\delta $  is the solution given in Theorem \ref{t2.11} and $\rho^\delta= 1-|\nabla \Phi^\delta |^2$.
It is easy to see that
\begin{equation}
S_1 [W_0]=\epsilon^2 \Delta \rho^\delta, \ \ S_2 [W_0 ]=0.
\end{equation}

We observe that for this initial approximate solution $W_0$, the second component satisfies the Neumann boundary condition but  the first component does not, which has to be corrected by a {\em boundary layer}. To this end, we now add a correction  function to the first component: let $ \rho_1 $ be the unique solution of
\begin{equation}
\label{s2.3}
\epsilon^2 \Delta \rho_1 -2 (\rho^\delta)^2 \rho_1=\epsilon^2 \Delta \rho^\delta \ \  \ \mbox{in} \ \R^d \backslash \Omega, \ \ \
\frac{\partial \rho_1}{\partial \nu }= -\frac{\partial \rho^\delta }{\partial \nu } \ \mbox{on}  \ \partial \Omega.
\end{equation}

Observe that by the estimates for  $|\nabla \Phi^\delta|$ in Lemma \ref{t2.11}, it holds that
\begin{equation}
 |\Delta \rho^\delta |\lesssim  \ <x>^{-3}.
 \end{equation}

On the other hand, using classical barrier function, we have the following estimates for $\rho_1 $:
\begin{equation}
\label{barrho}
\rho_1 \lesssim \  \epsilon e^{- C  d(x, \partial \Omega)/\epsilon} + \epsilon^2 <x>^{-3}, \ |\nabla \rho_1 | \lesssim \ e^{- C d(x, \partial \Omega)/\epsilon}+\epsilon^2 <x>^{-4}.
\end{equation}

Let us choose the second approximation as follows
\begin{equation}
 W_1= (\rho^\delta +\rho_1, \Phi^\delta  ).
 \end{equation}
We then compute
\begin{eqnarray}
S_1 [W_1] & = &  {\mathcal O}  (\rho_1^2)= {\mathcal  O} (\epsilon^2 e^{- C d(x, \partial \Omega)/\epsilon } +\epsilon^4 <x>^{-6}), \label{S1W}\\
 S_2 [W_1] & = & 2 \nabla (\rho_\delta \rho_1 \nabla \Phi^\delta) +\nabla (\rho_1^2 \nabla (\Phi^\delta ))=  \nabla \left( {\mathcal O} ( \epsilon e^{- C d(x, \partial \Omega)/\epsilon } +\epsilon^2 <x>^{-4} ) \right).
 \label{S2W}  \end{eqnarray}

Now we linearize around $W_1$ as follows
\begin{equation}
\label{rhophi}
 \rho_\epsilon = \rho^\delta + \rho_1 +\rho_2, \ \Phi_\epsilon=\Phi^\delta+ \epsilon \Phi_2
 \end{equation}
and  obtain
$$ S_1 [\rho_\epsilon, \Phi_\epsilon]= S_1 [W_1]+  L_1 (\rho_2, \Phi_2)+ N_1 (\rho_2, \Phi_2) $$
$$ S_2 [\rho_\epsilon, \Phi_\epsilon] = S_2 [W_1]+ \epsilon^{-1} L_2 (\rho_2, \Phi_2) + \epsilon^{-1} N_2 (\rho_2, \Phi_2) $$
where
$$ L_1 (\rho, \Phi)= \epsilon^2 \Delta \rho_2 -2 (\rho^\delta+\rho_1)^2 \rho_2 +2 \epsilon (\rho^\delta+\rho_1) \nabla \Phi^\delta \nabla \Phi_2,
$$
$$ L_2 (\rho, \Phi)= \epsilon^2 \nabla ((\rho^\delta+\rho_1)^2 \nabla \Phi_2)+ \epsilon \nabla (2 (\rho^\delta+\rho_1) \rho_2 \nabla \Phi^\delta),
$$
and
\begin{equation}
\label{N1N2}
\left\{\begin{array}{l}
 N_1= (\rho^\delta+\rho_1) (-\epsilon^2 |\nabla \Phi_2|^2) +\rho_2 (-2 \epsilon  \nabla \Phi^\delta \nabla \Phi_2 -\epsilon^2 |\nabla \Phi_2|^2 -3 (\rho^\delta+ \rho_1)\rho_2-\rho_2^2), \\
 N_2= \epsilon^2 \nabla (2 (\rho^\delta +\rho_1) \rho \nabla \Phi_2)+\epsilon \nabla (\rho^2 \nabla (\Phi^\delta +\epsilon \Phi_2)).
\end{array}
\right.
\end{equation}
Observe  that $N_2$ is of the form
\begin{equation}
 N_2= \epsilon \nabla (g), \   \ \mbox{where}\ \ \ \ g=  2\epsilon  (\rho_\delta +\rho_1)\rho_2  \nabla \Phi_2+\rho^2 \nabla (\Phi^\delta +\epsilon \Phi_2)).
 \end{equation}

An important observation is  that if $ \frac{\partial \rho }{\partial \nu}=\frac{\partial \Phi}{\partial \nu}=0$ on $ \partial \Omega$,  then it holds that
\begin{equation}
\label{gdef}
 g \cdot \nu =0 \ \mbox{on} \ \partial \Omega.
 \end{equation}

 We aim to solve the following system of equations
 \begin{equation}
 \label{neweqn}
 \left\{\begin{array}{l}
  L_1 (\rho_2, \Phi_2)= - E_{1,1}- N_1 (\rho_2, \Phi_2), \\
L_2 (\rho_2, \Phi_2) =-  E_{1,2}- N_2 (\rho_2, \Phi_2)=-\epsilon \nabla (g)
\end{array}
\right.
\end{equation}
where
 \begin{equation}
 \label{error00}
  E_{1,1}= {\mathcal O} (\rho_1^2)= {\mathcal O} (\epsilon^2 e^{- C d(x, \partial \Omega)/\epsilon } +\epsilon^4 <x>^{-6}),  \ E_{1,2}= \epsilon S_2 [W_1]= \epsilon \nabla \left( {\mathcal O}( \epsilon e^{- C d(x, \partial \Omega)/\epsilon } +\epsilon^2 <x>^{-4} ) \right).
 \end{equation}

The computations above provides basic estimates to proceed in the next steps.

\subsection{Norms and Errors}
We now introduce weighted Sobolev spaces. Let $ x=\epsilon y$ where $ y\in \R^d \backslash  \Omega_\epsilon$ with $\Omega_\epsilon= \frac{\Omega}{\epsilon}$. The operator ${\mathbb L}_1$ and $ {\mathbb L}_2$ becomes
 $$ L_1 (\rho, \Phi)=  \Delta_y \rho -2 (\rho_\delta+\rho_1)^2 \rho +2  (\rho_\delta+\rho_1) \nabla_x \phi_\delta \nabla_y \Phi
$$
$$ L_2 (\rho, \Phi)= \nabla_y ((\rho_\delta+\rho_1)^2 \nabla \Phi)+ \nabla (2 (\rho_\delta+\rho_1) \rho \nabla_x \Phi^\delta).
$$

 To this end,  we define the following weighted norms:
\begin{equation}
\| \rho\|_{*,1}=\sup_{y\in \R^d \backslash \Omega_\epsilon}  \|  \rho\|_{W^{2,4}  (B_1 (y) \cap \R^d \backslash \Omega_\epsilon) } + \| \rho \|_{W^{1,2} (\R^d \backslash \Omega_\epsilon)},
\end{equation}
\begin{equation}
\| \Phi\|_{*,2}=\| \nabla \Phi \|_{L^{2} (\R^d \backslash \Omega_\epsilon)}+\| \nabla \Phi \|_{L^{4} (\R^d \backslash \Omega_\epsilon)}
\end{equation}
\begin{equation}
\| f\|_{**,1}= \| f \|_{L^2 (\R^d \backslash \Omega_\epsilon)},
\end{equation}
\begin{equation}
\| g\|_{**,2}= \| g \|_{L^2 (\R^d \backslash \Omega_\epsilon)}+\| g \|_{L^4 (\R^d \backslash \Omega_\epsilon)}.
\end{equation}

 Since
$$ \epsilon^4 <x>^{-6} \lesssim\  \epsilon <\frac{|x|}{\epsilon}>^{-3}, \ \ \ \epsilon^2 <x>^{-4} \lesssim\   \epsilon^{1/2}  <\frac{|x|}{\epsilon}>^{-\frac{3}{2}},
$$
 from (\ref{error00}), we derive that
\begin{equation}
\label{error2m}
\| E_{1,1} \|_{**,1 } + \| E_{1,2} \|_{**,2} \lesssim \ \epsilon^{\sigma}.
\end{equation}
where $ \beta = 1+\sigma$ for some $\sigma \in (0, 1)$. (In fact we may choose $\sigma=\frac{1}{2}$.)

\subsection{A priori estimates} To proceed with the perturbation, we need the following important {\em a priori estimates}.

\begin{theorem}
\label{apriori}
For $\epsilon$ sufficiently small and for  each $(f, g)$ with $ \| f\|_{**,1}+\| g\|_{**,2} <+\infty$, there exists a pair $(\rho, \phi)$ satisfying
\begin{equation}
\label{apriori1}
\left\{\begin{array}{l}
L_1 [\rho, \Phi] = f, L_2 [\rho, \Phi]=\nabla g, \ \  \mbox{in} \ \R^d \backslash \Omega,\\
  \ \frac{\partial \rho}{\partial \nu}=0,   \frac{\partial \Phi}{\partial \nu}= 0 \  \ \mbox{on}\  \partial \Omega.
\end{array}
\right.
\end{equation}
Furthermore
\begin{equation}
\label{2p}
\| \rho \|_{*,1} + \| \Phi\|_{*,2 } \lesssim \| f\|_{**,1 }+ \|  g \|_{**, 2}.
\end{equation}
\end{theorem}

\begin{proof}
We first prove a global $L^2$ estimate. Namely if $ \| f\|_{L^2}+ \| g \|_{L^2}<+\infty$ then we have
\begin{equation}
\label{A11}
\int_{\R^d \backslash \Omega_\epsilon } (|\nabla \rho|^2+ |\nabla \Phi |^2+\rho^2) \lesssim  \int_{\R^d \backslash \Omega_\epsilon } (|f|^2+|g|^2).
\end{equation}

In fact multiplying the first equation in (\ref{apriori1}) by $ \rho$  we obtain
$$ \int_{\R^d \backslash \Omega_\epsilon } (|\nabla \rho|^2+ 2 (\rho_0+\rho_1)^2 \rho^2) = 2\int_{\R^d \backslash \Omega_\epsilon }  (\rho_\delta +\rho_1)\nabla_x |\Phi^\delta| \nabla_y \Phi  \rho -  \int_{\R^d \backslash \Omega_\epsilon } \rho f
$$
and hence
\begin{equation}
\label{A12}
2 \int_{\R^d \backslash \Omega_\epsilon } ( (\rho_0+\rho_1)^2 \rho^2) \leq  2 \sup \frac{ |\nabla_x \Phi^\delta|^2 }{ (\rho_0+\rho_1)^2 } \int_{\R^d \backslash \Omega_\epsilon } (\rho_\delta+\rho_1)^2  |\nabla \Phi |^2 + C \int_{\R^d \backslash \Omega_\epsilon } f^2
\end{equation}
Multiplying the second equation in (\ref{apriori1}) by $\Phi$ and using (\ref{gdef}) we obtain
\begin{equation}
\label{A13}
\int_{\R^d \backslash \Omega_\epsilon } ( (\rho_\delta+\rho_1)^2 |\nabla \rho|^2) \leq 2 \int_{\R^d \backslash \Omega_\epsilon } (\rho_\delta +\rho_1) \rho \nabla \Phi^\delta \nabla \Phi  + \int_{\R^d \backslash \Omega_\epsilon } g \nabla \Phi.
\end{equation}
Substituting (\ref{A13}) into (\ref{A12}) and using the fact that $ 2 \sup \frac{ |\nabla \Phi^\delta|^2 }{ (\rho_0+\rho_1)^2 } <1$, we obtain the apriori estimates (\ref{A11}).

To finish the proof of a priori estimates, we use elliptic regularity theory.  Since  $ \rho \in W^{2,2}(\R^2 \backslash \Omega_\epsilon)$ and $d\leq 2,3$,  we have that $\nabla \rho \in L^4 (\R^2 \backslash \Omega_\epsilon)$. By $L^p-$ estimates for divergence operators (see \cite{BW}), we also obtain $\nabla \Phi \in L^4 (\R^2 \backslash \Omega_\epsilon)$. (Here the condition (\ref{gdef}) is used.)

From the a priori estimates (\ref{A11}) and standard degree argument we obtain the existence of the system (\ref{apriori1}) in any bounded domain $ B_{R} \backslash \Omega_\epsilon$ for $R$ large, coupled with Dirichlet boundary condition
$$ \rho=\Phi=0 \ \mbox{on} \ \  \partial B_R.$$
Then letting $R\to +\infty$ we obtain a solution satisfying (\ref{2p}).

\end{proof}

\subsection{Proof of Theorem \ref{t1}}

Theorem \ref{t1} can be now proved by a contraction mapping argument to solve (\ref{neweqn}). Let
$$ \|(\rho_2, \Phi_2)\|_{*}:= \| \rho_2 \|_{*, 1} + \| \Phi_2\|_{*,2 }$$
$$ \| (f, g) \|_{**} :=  \| f\|_{L^2(\R^d \backslash \Omega_\epsilon)  }+ \| g \|_{L^2 (\R^d \backslash \Omega_\epsilon)} +\| g \|_{L^4 (\R^d \backslash \Omega_\epsilon)}.
$$

 In fact by the estimates (\ref{error2m}) we have
\begin{equation}
\| (E_{1,1}, E_{1,2}) \|_{* } \lesssim \epsilon^\sigma.
\end{equation}

Let $ (\rho_2, \Phi_2)$ be such that $ \| (\rho, \Phi)\|_{*} \lesssim \epsilon^\sigma$. Let us estimate the nonlinear terms. In the rescaled variable $x=\epsilon y$, we have
$$
N_1= (\rho_\delta+\rho_1) (- |\nabla_y \Phi|^2) +\rho_2  (-2  \nabla \Phi^\delta \nabla_y \phi_2 - |\nabla_y \phi_2 |^2 -3 (\rho_\delta+ \rho_1)\rho_2-\rho_2^2).
$$
Since $ \nabla \Phi_2 \in L^4, \rho_2 \in  W^{2,2}$, we see that for $d=2,3$, $ |\rho_2 |_{L^\infty} \lesssim \epsilon^\sigma$ and that
\begin{equation}
\label{N1est}
 \| N_1\|_{L^2} \lesssim \epsilon^{2\sigma}.
 \end{equation}
Similarly in the rescaled variable
$$ N_2=  \nabla_y (2 (\rho_\delta +\rho_1) \rho \nabla_y \Phi_2 )+ \nabla (\rho^2 \nabla_y (\phi_\delta +\epsilon \Phi_2 ))=\nabla_y (g)$$
with
\begin{equation}
\label{N2est}
 \| g\|_{L^2 }+\| g\|_{L^4} \lesssim \epsilon^{2\sigma}.
 \end{equation}

From (\ref{N1est}) and (\ref{N2est}) and a  standard contraction mapping we obtain Theorem \ref{t1}. \qed

\subsection{Remarks on further estimates of $(\rho_2, \Phi_2)$ near the boundary}

For later purpose, we need to expand $ (\rho_2, \Phi_2)$ near a boundary point. It turns out that the estimate is better. Let $d=2$ and $ x_0 \in \partial \Omega$. We may assume that the normal direction is $ \nu_{x_0}= \vec{e}_1$, the tangential direction is  $\tau_{x_0} = \vec{e}_2$ and $ \R^d \backslash \Omega \subset \{ x_1 \leq 0 \}$. Rescale $ x= x_0+\epsilon \bar{y}$. Then we see that
\begin{equation}
\label{rho1}
\rho_1 (x_0+\epsilon \bar{y}) \sim  \epsilon  (-\frac{\partial \rho^\delta}{\partial \nu} (x_0) )  e^{ \sqrt{2} \rho^\delta (x_0) \bar{y}_1} + {\mathcal O} (\epsilon^2 e^{ C \bar{y}_1} ).
\end{equation}

With (\ref{rho1}) we can give better estimates on $\epsilon S_2 [W_1]$. Note that
$$ \epsilon S_2 [W_1] \sim \epsilon \nabla (\rho^\delta \rho_1 \nabla \Phi^\delta)  + \epsilon \nabla ( {\mathcal  O} ( \epsilon^2) e^{C\bar{y}_1} ) $$
For the first term, we have
$$  \rho^\delta \rho_1 \nabla \Phi^\delta  \sim \epsilon (-\frac{\partial \rho^\delta}{\partial \nu} (x_0) )  e^{ \sqrt{2} \rho^\delta (x_0) \bar{y}_1}   \rho^\delta (x_0) |\nabla \Phi^\delta (x_0)| \vec{e}_2 $$
Thus
$$ \epsilon S_2 [W_1] \sim \epsilon \nabla  (  \epsilon (-\frac{\partial \rho^\delta}{\partial \nu} (x_0) )  e^{ \sqrt{2} \rho^\delta (x_0) \bar{y}_1}   \rho^\delta (x_0) |\nabla \Phi^\delta (x_0)| \vec{e}_2 ) + \epsilon \nabla ( {\mathcal O}( \epsilon^2)e^{C\bar{y}_1}).
$$
This implies that
\begin{equation}
\epsilon S_2 [W_1] \sim \epsilon \nabla ( {\mathcal O} ( \epsilon^2)e^{C\bar{y}_1})
\end{equation}

Since the remaining errors in  (\ref{error00}) carries at least $\epsilon^2$ order, we conclude that
\begin{equation}
 \label{error000}
 E_{1,1}= {\mathcal O} (\epsilon^2 e^{-\frac{ C d(x, \partial \Omega)}{\epsilon}} +\epsilon^4 <x>^{-6}),  \ \ E_{1,2}= \epsilon \nabla ( \epsilon^2 e^{- C d(x, \partial \Omega)/\epsilon } +\epsilon^2 <x>^{-4} ).
 \end{equation}

Therefore we may take $ \sigma= 1+\sigma_0 \in (1,2)$  in the proof of Theorem \ref{t1} to obtain that
\begin{equation}
\label{320}
 \rho_\epsilon = \rho^\delta + {\mathcal O} (\epsilon^{1+\sigma_0}), \Phi_\epsilon =\Phi^\delta + {\mathcal O} (\epsilon^{2+\sigma_0} ).
 \end{equation}

\setcounter{equation}{0}
\section{Derivation and Mapping properties of Traveling Waves}

\subsection{Derivation of traveling wave equation}
Let $u_\epsilon=\rho_\epsilon e^{i \frac{\Phi_\epsilon}{\epsilon}}$ be the solution constructed in Theorem \ref{t1}.  By the remark at the last section, we have  $
 \rho_\epsilon = \rho^\delta + {\mathcal O} (\epsilon^{1+\sigma_0} ), \Phi_\epsilon =\Phi^\delta + {\mathcal O} (\epsilon^{2+\sigma_0} ) $ for some $\sigma_0>0$. (In the computations below we may simply assume that $ \rho_\epsilon=\rho^\delta, \Phi_\epsilon= \Phi^\delta.$)

Now we look for another solution of the following form
\begin{equation}
u=\rho_\epsilon e^{i \frac{\Phi_\epsilon}{\epsilon}}\ v.
\end{equation}
We see that $v$ satisfies
\begin{equation}
\label{3.2n}
\left\{\begin{array}{l}
\epsilon^2 \Delta v+2\epsilon^2  \nabla \log \rho_\epsilon \nabla v+ 2i \epsilon \nabla \Phi_\epsilon \nabla v+ v \rho_\epsilon^2 (1-|v|^2)
=0 \ \mbox{in} \ \R^d \backslash \Omega, \\
\ \frac{\partial v}{\partial \nu }=0\ \mbox{on} \ \partial \Omega.
\end{array}
\right.
\end{equation}

In the following, we explain the intuitive  idea in the construction of the second solution. Let $x_0 \in \partial \Omega$ (to be determined later) and we perform a blow-up near $x_0$: let $x=x_0+\epsilon y$. We see that $\nabla \Phi_\epsilon (x)=\nabla \Phi^\delta (x_0)+ {\mathcal O} (\epsilon^{1+\sigma})$ (see (\ref{320})). (Without loss of generality we may assume that $ \nabla \Phi^\delta (x_0) = |\nabla \Phi^\delta (x_0)| e_d$.) Since $ \rho_\epsilon (x)=\rho^\delta (x)+ {\mathcal O} (\epsilon)=\rho^\delta (x_0) +{\mathcal O} (\epsilon |y| +\epsilon^\sigma)$,  we see that formally the limit equation  for (\ref{3.2n}) becomes
\begin{equation}
\label{3.2}
\Delta v +2 i |\nabla \Phi^\delta (x_0)| \frac{\partial v}{\partial y_d} + v (1-|\nabla \Phi^\delta (x_0)|^2) (1-|v|^2)=0 .
\end{equation}
Changing $y= \frac{ \bar{y}}{\sqrt{1-|\nabla \Phi^\delta (x_0)|^2}}$, $ U(y)= v(\bar{y})$ we see that (\ref{3.2}) is equivalent to
\begin{equation}
\label{3.2v}
\Delta_{\bar{y}}  U+ic \frac{\partial U}{\partial \bar{y}_d} + U  (1-|U|^2)=0 \ \ \mbox{in} \ \R^d
\end{equation}
with $c=\frac{2 |\nabla \Phi^\delta (x_0)| }{\sqrt{1-|\nabla \Phi^\delta (x_0)|^2}}$. Observe that
\begin{equation}
0<c <\sqrt{2} \ \iff \ |\nabla \Phi^\delta (x_0)|^2 <\frac{1}{3}
\end{equation}
which is equivalent to the subsonic assumption (\ref{phi1})  in Theorem \ref{t2.11}.

Problem (\ref{3.2v}) arises as the traveling wave solution  $ u(\bar{y}-c t \vec{e}_d) $ for the Gross-Pitaevskii equation
\begin{equation}
i u_t= \Delta u+ u(1-|u|^2) \ \ \ \ \ \ \ \ \ \ \ \ \mbox{in} \ \R^d.
\end{equation}

It has been proved that a necessary condition for the existence of traveling waves is $ |c|<\sqrt{2}$. In the case of small speed $|c|<<1$, the existence of traveling vortex  ($d=2$) and vortex rings ($d\geq 3$) has been proved by Bethuel-Saut \cite{BS}, Bethuel-Orlandi-Smets \cite{BOS}, and \cite{lw-cpam}. The method of \cite{BS, BOS} is variational while
we used a perturbation approach which is related to what we will employ in this paper. The asymptotics of traveling wave solutions is studied in \cite{grav1}-\cite{grav4}. For general speed $c$, we refer to Bethuel-Gravajat-Saut \cite{BGS}, Gravajat \cite{grav1, grav2, grav3, grav4} and references therein.

\subsection{Properties of traveling waves solutions}

In this sub-section, we are concerned with the properties of the traveling wave solution to (\ref{3.2v}) in $\R^2$. So from now on, we assume that $d=2$ and we consider
\begin{equation}
\label{limit3.1}
\Delta U+ic \frac{\partial U}{\partial y_2} + U  (1-|U|^2)=0 \ \ \mbox{in} \ \R^2.
\end{equation}

  In  \cite{lw-cpam}, we used a perturbation approach to prove the existence of a traveling wave solution to (\ref{limit3.1}) with two opposite vortices traveling in the direction $y_2$. We summarize the  result in the following

\begin{lemma}
\label{property}
(\cite{lw-cpam})
For $c$ sufficiently small, there exists a solution $U_c$ to (\ref{limit3.1}) with the following properties

\noindent
(i) $U_c ( y )= w^{+} (y-d_c) w^{-} (y+d_c) + {\mathcal O} (|c|)$, where $ w^{\pm} (y)$ is the unique degree 1 (or $-1$) solution of Ginzburg-Landau equation
\begin{equation}
\label{limit3.11}
\Delta U+ U  (1-|U|^2)=0 \ \ \mbox{in} \ \R^2,
\end{equation}
and
 \begin{equation}
 d_c= ({\mathcal O} (\frac{1}{c}), 0);
 \end{equation}

 \noindent
 (ii) $U_c$ is even in $y_1$, i.e.,
 \begin{equation}
 \frac{\partial U_c}{\partial y_1}=0 \ \ \mbox{on} \ y_1=0;
 \end{equation}

\noindent
(iii) Writing $U_c (y)= S(y) e^{ i \varphi (y)}$, then it holds that
\begin{equation}
\label{32}
|\nabla S (y)| \leq \frac{ C(c)}{ |y- y_c|^3}, |\nabla \varphi (y)| \leq \frac{C(c)}{|y-y_c|^2}
\end{equation}
where $C(c) \lesssim |c|^{-2}$.
 \end{lemma}

\begin{proof}
Properties (i)-(iii) follow from the constructions given in \cite{lw-cpam}. (In \cite{lw-cpam}, Schrodinger map is studied. But the same computations work for  (\ref{limit3.1}) as well.) For the asymptotics (\ref{32}), it also follows from \cite{grav1}.

\end{proof}

The following theorem give a complete characterization of the kernels of the linearized operator, at least when the speed $c$ is small. The proof of it will be delayed to Section \ref{sec6}.

\begin{theorem}
\label{t3.1}
Let $U_c$ be the solution constructed in \cite{lw-cpam}. Consider the linearized operator
\begin{equation}
\label{ker10}
{\mathbb L}_0 [\phi]= \Delta \phi + i c\frac{\partial \phi}{\partial y_2} +  (1- |U_c|^2) \phi -2 (U_c \cdot \phi) U_c.
\end{equation}
Then for $c$ sufficiently small, the only bounded kernels satisfying
\begin{equation}
\label{ker1}
\left\{\begin{array}{l}
{\mathbb L}_0 [\phi]=0 \ \ \mbox{in} \ \R_{+}^2= \{ (y_1, y_2) | y_1>0 \}, \\
 \frac{\partial \phi }{\partial y_1}=0 \ \ \mbox{on} \ \partial \R^2_{+}, \\
 \| \phi\|_{L^\infty (\R_{+}^2)} <+\infty,
 \end{array}
 \right.
 \end{equation}
are
\begin{equation}
\phi= \beta_1 (i U_c)+  \beta_2 \frac{\partial U_c}{\partial y_2}
\end{equation}
where $\beta_1$ and $ \beta_2$ are some constants.
\end{theorem}

In the general large $c$ case, the following result is proved in several papers (\cite{BS}-\cite{BGS}).

\begin{theorem}

\label{t3.3}

\noindent
(1) If $c \geq \sqrt{2}$, there are no traveling wave solutions to (\ref{limit3.1});

\noindent
(2) For $ 0<c< \sqrt{2}$, there exists a solution to (\ref{limit3.1}) which is the ground state solution among the fixed moment;

\noindent
(3) Let $ c<\sqrt{2}$, and $U_c$ be a solution to (\ref{limit3.1}. Then as $ |y| \to +\infty$,
\begin{equation}
\label{Uclimit}
 U_c- 1= O( <y>^{-1})
 \end{equation}

\end{theorem}

\begin{proof}
(1) and (2) are proved in \cite{bgs} and (3) is proved in \cite{grav1}. (Note that the exact asymptotics of $U_c-1$ is not precised in dimension two.)
\end{proof}

By Theorem \ref{t3.1}, there exists a nondegenerate solution $U_c$ for $c$ small. From now, we assume the following:

\medskip

\noindent
{\bf Key Assumption:} {\em There exists a nondegenerate solution $U_c$ for $c \in (0, c_0)$ for some $ c_0 \in (0, \sqrt{2}]$. This means that there exists a solution to (\ref{limit3.1}) such that the only solutions to the linearized problem  (\ref{ker1})
are
\begin{equation}
\label{ker12}
\phi= \beta_1 (i U_c)+  \beta_2 \frac{\partial U_c}{\partial y_2}
\end{equation}
where $\beta_1$ and $ \beta_2$ are some constants.
}

\begin{remark}
For $c$ small, there are higher energy solutions with $ l(l+1)$ traveling vortices. These vortices are located at the  roots of  Adler-Moser polynomials. The existence and properties of these multiple vortices solutions are considered in \cite{LiuWei}. We believe that these higher energy solutions are also nondegenerate, which  may provide new solutions to equation (\ref{1.1}).
\end{remark}

\subsection{Perturbation of traveling wave solution and sketch of proof of Theorem \ref{t2}}

Under the {\bf Key Assumption}, in the remaining sections,  we use a finite dimensional reduction method to rigorously prove the existence of such solutions to (\ref{3.2n}) and hence give the proof of Theorem \ref{t2}. The important step is the determination of $x_0 \in \partial \Omega$.

We solve (\ref{3.2n}) in three steps.

\medskip

\noindent
{\bf Step 1:}  Fixing each $x_0 \in \partial \Omega$ we show that there exists a unique solution to (\ref{3.2n}) and two Lagrange multipliers $\lambda_0, \lambda_1$ such that
\begin{equation}
\label{3.6}
\left\{\begin{array}{l}
  \epsilon^2 \nabla (\rho_\epsilon^2 \nabla v) + 2i  \epsilon \rho_\epsilon^2 \nabla \Phi_\epsilon \nabla v+ v \rho_\epsilon^4 (1-|v|^2)
=\lambda_0 i \rho_\epsilon^2 Z_0 + \lambda_1 \rho_\epsilon^2 Z_1 \ \mbox{in} \ \R^2 \backslash \Omega, \\
\ \frac{\partial v}{\partial \nu }=0\ \mbox{on} \ \partial \Omega,  \\
v \sim U_c
\end{array}
\right.
 \end{equation}
where $ Z_0$ and $ Z_1$ are defined at (\ref{Zd}).

\medskip

\noindent
{\bf Step 2:} We solve the following algebraic equation
\begin{equation}
 \lambda_1=0.
\end{equation}
By asymptotic expansion, we find that
\begin{equation}
 \lambda_1 \sim \frac{\partial  }{\partial \tau_{x_0}} (|\nabla \Phi^\delta |^2 (x_0)).
\end{equation}
where $\tau_{x_0}$ denotes the tangential direction at $x_0 \in \partial \Omega$.

By placing $x_0$ near the maximum or the minimum of $ |\nabla \Phi^\delta |^2 (x_0)$ on the $\partial \Omega$, we can adjust $x_0$ such that $ \lambda_1 =0$. (Hence we always obtain at least {\em two} solutions, as stated in Theorem \ref{t2}.)

\medskip

\noindent
{\bf Step 3:} In the last step we use the gauge invariance of the equation to show that $\lambda_0=0$.

\medskip

In the following sections, we carry out this procedure. At first we need to understand the error and invertibility of the linearized operator.

\setcounter{equation}{0}
\section{Proof of Theorem \ref{t2}}

In this section we follow the steps outlined above to prove Theorem \ref{t2}.

\subsection{Approximate solutions and error estimates}

Let $x_0 \in \partial \Omega$ and $ x= x_0+\epsilon y $. ($x_0$ is {\em a priori} undetermined.) Without loss of generality we  also assume that the outer normal direction $ \nu_{x_0}=  \vec{e}_1$ and the tangential direction $ \tau_{x_0}= \vec{e}_2$. In this stretched variable equation (\ref{3.2n}) is transformed to 
\begin{equation}
\label{4.1}
\Delta_y v+ 2\epsilon  \nabla_x \log \rho_\epsilon \nabla_y v + 2i   \nabla_x \Phi_\epsilon \nabla v+   \rho_\epsilon^2 v (1-|v|^2)=0.
\end{equation}
Here the new stretched domain becomes
$$ y\in \R^2 \backslash \Omega_{\epsilon, x_0}, \ \ \mbox{where} \ \Omega_{\epsilon, x_0}= \frac{1}{\epsilon} (\Omega -\{ x_0\}) $$
and
$$ \frac{\partial v}{\partial \nu}=0 \ \ \mbox{on} \ \partial \Omega_{\epsilon, x_0}. $$

We write $ W (y):= U_c (\bar{y})= S(\bar{y}) e^{i \varphi (\bar{y})} $ where $U_c$ is given by Theorem \ref{t3.1} and $ \bar{y}= \sqrt{ 1- |\nabla \Phi^\delta (x_0)|^2} \ y$.  (Note that $S(\bar{y})$ is not radially symmetric.)  To avoid clumsy notations, we also drop the dependence on the speed $c$ and  the location $x_0$.
Then $W$ satisfies
\begin{equation}
\label{4.2}
\Delta_y  W + i |\nabla \Phi^\delta (x_0)|  \frac{\partial W}{\partial y_2 } + (\rho^\delta (x_0))^2 W (1-|W|^2)=0.
\end{equation}
By Theorem \ref{t3.3}, we have the following decaying estimate for $S$ and $\varphi$:
\begin{equation}
\label{Svardecay}
\nabla S= {\mathcal O} ( <y>^{-3}), \  \ \nabla \varphi ={\mathcal  O} (<y>^{-2}).
\end{equation}

We write (\ref{4.1})  as a solution operator
\begin{equation}
\label{2.00}
{\mathbb S} [v]=
\Delta v +
+ v \rho_\epsilon^2 (1-|v|^2) + 2\epsilon  \nabla_x \log \rho_\epsilon \nabla_y v+ 2 i   \nabla_x  \Phi_\epsilon \cdot   \nabla_y v.
\end{equation}
Using (\ref{4.2}) we get
\begin{equation}
\label{SW1}
{\mathbb S} [W]=2\epsilon  \nabla_x \log \rho_\epsilon \nabla_y W+ 2 i   (\nabla_x  \Phi_\epsilon - |\nabla \Phi^\delta (x_0)|  \vec{e}_2) \cdot   \nabla_y W+ (\rho_\epsilon^2 - (\rho^\delta (x_0))^2) W (1-|W|^2).
\end{equation}

In this section, we estimate the size of the error ${\mathbb S} [W]$. The first term in (\ref{SW1}) is expanded as

 \begin{equation}
 \label{4.8}
 2\epsilon \nabla_x \log \rho_\epsilon \nabla W= 2\epsilon \nabla_x \log \rho_\epsilon \nabla S e^{i \varphi} + i2\epsilon \nabla_x \log \rho_\epsilon \nabla \varphi   S e^{i \varphi}.
 \end{equation}

 For the first and second terms in (\ref{4.8}), we make use of the decaying estimate (\ref{Svardecay}) and decaying estimate (\ref{rhophidecay}) . The first term has the following decay estimates
\begin{equation}
\label{34}
2\epsilon |\nabla \log \rho_\epsilon \nabla S|\lesssim \epsilon <y>^{-3}
\end{equation}
while for  the second term gives
\begin{equation}
\label{35}
2\epsilon |\nabla \log \rho_\epsilon \nabla \varphi |\lesssim \epsilon <\epsilon y>^{-2} <y>^{-2} \lesssim \epsilon^{1-\sigma} (1+|y|)^{-2-\sigma }
\end{equation}
where $ \sigma \in (0, 1)$ is any given positive number.

For the second term in (\ref{4.8}), we have
\begin{eqnarray*}
& & 2i \nabla_x \Phi_\epsilon \cdot \nabla_y W \\
 &= & 2i |\nabla \Phi^\delta (x_0)| \vec{e}_2 \cdot \nabla_y W + 2i (\nabla_x \Phi^\epsilon -|\nabla \Phi^\delta (x_0)|  e_2) \nabla W \\
& = & 2i |\nabla \Phi^\delta (x_0)|  \vec{e}_2 \cdot \nabla_y W + 2i (\nabla_x \Phi^\epsilon -|\nabla \Phi^\delta (x_0)|  e_2) \nabla S e^{i\varphi } - 2(\nabla_x \phi_\epsilon -|\nabla \Phi^\delta (x_0)|  \vec{e}_2) \nabla \varphi S.
\end{eqnarray*}

By the decaying estimates in (\ref{Svardecay}) we get that
\begin{equation}
\label{error1}
{\mathbb S} [W] =  {\mathcal O} (\epsilon^{1-\sigma} <y>^{-2-\sigma}) + 2i (\nabla_x \Phi_\epsilon -|\nabla \Phi^\delta (x_0)|  e_2) \nabla S e^{i\varphi } - 2(\nabla_x \Phi_\epsilon -|\nabla \Phi^\delta (x_0)|  \vec{e}_2) \nabla \varphi S   e^{i \varphi}.
 \end{equation}

 We note that for $ |y| >>1$
 \begin{equation}
 \label{error21}
 \frac{{\mathbb S} [W] }{iW}= {\mathcal O} (\epsilon^{1-\sigma} <y>^{-2-\sigma}) + 2 (\nabla_x \Phi_\epsilon -|\nabla \Phi^\delta (x_0)| e_2) \nabla \log S + 2 i(\nabla_x \Phi_\epsilon -|\nabla \Phi^\delta (x_0)|  \vec{e}_2) \nabla \varphi
 \end{equation}

Using (\ref{Svardecay}) again we obtain the following basic error estimates
\begin{equation}
 \label{error3}
Re ( \frac{{\mathbb S} [W] }{i W})={\mathcal  O}(\epsilon^{1-\sigma } <y>^{-2-\sigma}), \ Im ( \frac{{\mathbb S} [W] }{ i W})= O(\epsilon^{1-\sigma} <y>^{-1-\sigma}).
 \end{equation}

Concerning the boundary behavior, we can strengthen the boundary and  use the fact that $ \frac{\partial W}{\partial y_1} (0, y_2)=0$ as well as the decaying (\ref{Svardecay}) to deduce that
\begin{equation}
\label{bdry}
Re( \frac{1}{iW} \frac{\partial W}{\partial \nu})= {\mathcal O} (\epsilon^\sigma <y>^{-1-\sigma} ),  \ Im ( \frac{1}{iW} \frac{\partial W}{\partial \nu})= {\mathcal O} (\epsilon^\sigma <y>^{-2-\sigma} ) \ \mbox{on} \ \partial \Omega_{\epsilon, x_0}.
\end{equation}

\subsection{Equations in operator form}

We look for solutions of (\ref{4.1}) in the form of a small perturbation of $W$. Let $\eta$ be a  smooth cut-off function such that $\eta (y)=1$ for $ |y| >R$ and $ \eta (y)=0$ for $|y|>2 R$. (Here $R$ is a large and fixed constant.)  As in \cite{DKM, DWZ} or \cite{lw-cpam} we look for solutions of (\ref{4.1}) of the form
\begin{equation}
v= \eta (W+ i W \psi) + (1-\eta) W e^{i \psi}
\end{equation}
where $ W \psi$ is small. We write $\psi=\psi_1+i \psi_2$ with $\psi_1$ and $\psi_2$ real-valued.

Set
$$
v= W +\phi, \ \ \ \  \phi= \eta i W \psi + (1-\eta) W (e^{i \psi} -1).
$$
For $ |y| >2R$,  the equation for $\phi$ becomes
\begin{equation}
{\mathbb L}_{\epsilon} [\phi] + {\mathbb N}_\epsilon  [\phi] = -{\mathbb S} [W]
\end{equation}
where
$$ {\mathbb L}_{\epsilon} ={\mathbb L}_{\epsilon,1} +{\mathbb L}_{\epsilon,2},$$
$$ {\mathbb L}_{\epsilon,1}=\Delta \phi +  \rho_\epsilon^2  ( (1-|W|^2)\phi -2 W \cdot \phi W) + 2i \rho_\epsilon^2 \nabla_x \Phi_\epsilon (x_0) \nabla \phi,
$$
$$
{\mathbb L}_{\epsilon, 2}= 2 \epsilon \nabla_x \rho_\epsilon \nabla \phi + 2 i (\nabla_x \Phi_\epsilon (x_0+\epsilon y) -\nabla_x \Phi_\epsilon (x_0)) \nabla \phi,
$$
$$
{\mathbb N}_\epsilon [\phi]=   \rho_\epsilon^2 [(1- |W+\phi|^2)( W+\phi)- (1- |W|^2)W - (1-|W|^2)\phi +2 W \cdot \phi W]
$$

For $ |y| >2R$, we have
\[ v= W e^{i \psi}\]
and thus by simple computations we get
\begin{eqnarray*}
  \frac{ {\mathbb S} [W e^{i \psi}]}{i W e^{i \psi}} &= &  \Delta \psi -2i \rho_\epsilon^2 |W|^2 \psi_2 + \frac{1}{W} \nabla W \cdot \nabla \psi  \\
  & & + 2 \epsilon \nabla_x \log \rho_\epsilon \cdot \nabla \psi +2 i \nabla_x \Phi_\epsilon \cdot \nabla \psi  \\
 & & + i \nabla \psi \cdot \nabla \psi  +i \rho_\epsilon^2 |W|^2 ( e^{-2\psi_2} -1+2\psi_2)  \\
& &  +\frac{ {\mathbb S} [W] }{ iW e^{i\psi}}
\end{eqnarray*}
which in terms of $\psi$ can be written as 
\begin{equation}
\label{L0}
\tilde{L}_{\epsilon, 1} [\psi] + \tilde{L}_{\epsilon, 2} [\psi]+ \tilde{N}_{\epsilon}  [\psi]=\tilde{E}
\end{equation}
where
\begin{eqnarray*}
 \tilde{L}_{\epsilon, 1} [\psi] &=& \Delta \psi -2i \rho_\epsilon^2 |W|^2 \psi_2 + \frac{1}{W} \nabla W \cdot \nabla \psi + 2i \nabla_x \Phi_\epsilon (x_0) \cdot \nabla \psi,
\\
 \tilde{L}_{\epsilon, 2}  [\psi] &= &    2i (\nabla_x \Phi_\epsilon (x) -\nabla_x \Phi_\epsilon (x_0)) \cdot \nabla \psi+ 2 \epsilon \nabla_x \log \rho_\epsilon (x) \cdot \nabla \psi,
\\
 \tilde{N}_\epsilon  [\psi] &=&  i \nabla \psi \cdot \nabla \psi  +i \rho^2 |W|^2 ( e^{-2\psi_2} -1+2\psi_2), \\
 \tilde{E} &=& -\frac{{\mathbb  S} [W] }{i We^{i\psi} }.
\end{eqnarray*}

Recalling that $\psi=\psi_1+i \psi_2$ and $ x= x_0 +\epsilon  y$, we have
$$  \frac{1}{W} \nabla W \cdot \nabla \psi= (\frac{1}{S} \nabla S+\frac{1}{S} i \nabla \varphi) \cdot \nabla \psi = {\mathcal O}( <y>^{-3} |\nabla \psi|) + i {\mathcal O}( <y>^{-2} |\nabla \psi|). $$

Hence
\begin{equation}
\label{L01}
\tilde{L}_{\epsilon, 1} [\psi]= \begin{pmatrix}
  \Delta \psi_1 - 2 \nabla_x \Phi_\epsilon (x_0) \cdot \nabla \psi_2+ O(<y>^{-2} ) |\nabla \psi|)  \\
\Delta \psi_2 -2\rho_\epsilon^2 |W|^2  \psi_2 + 2 \nabla_x \Phi_\epsilon (x_0) \cdot \nabla \psi_1  + O(<y>^{-2})|\nabla \psi|
\end{pmatrix}
\end{equation}
\begin{equation}
\label{N1n}
\tilde{L}_{\epsilon,2} [\psi]=\begin{pmatrix}
   O( \nabla_x \phi_\epsilon (x) \cdot \nabla \psi_2) + O(\epsilon \nabla_x \log \rho_\epsilon (x) \cdot \nabla \psi_1)
\\
  O( \nabla_x \phi_\epsilon (x) \cdot \nabla \psi_1) + O(\epsilon \nabla_x \log \rho_\epsilon (x) \cdot \nabla \psi_2)
\end{pmatrix}
\end{equation}
\begin{equation}
\label{N2}
\tilde{N}_\epsilon  [\psi]=\begin{pmatrix}
  -2 \nabla \psi_1 \cdot \nabla \psi_2
\\
  |\nabla \psi_1|^2-|\nabla \psi_2|^2 + O( e^{-2\psi_2} -1+2\psi_2 )
\end{pmatrix}
\end{equation}

$ \tilde{L}_{\epsilon, 2}$ contains {\em linear} terms which will be shown to be higher order, while $\tilde{N}_\epsilon $ contains {\em nonlinear} terms in $\psi$.
Let us remark that the explicit form of all the linear and nonlinear terms will be very useful for later analysis.

Finally equation (\ref{L0})  has to be solved with the following boundary condition
\begin{equation}
\label{vbc}
\frac{\partial \psi }{\partial \nu}=-\frac{1}{iW} \frac{\partial W}{\partial \nu } = {\mathcal O} (\epsilon^\sigma <y>^{-1-\sigma} )  + i {\mathcal O} (\epsilon^\sigma <y>^{-2-\sigma}) \ \ \mbox{on} \ \Omega_{\epsilon, x_0}.
\end{equation}

\subsection{Weighted norms and error estimates}

We first introduce some norms. Let us fix two small positive numbers $ 0< \gamma <1,  0<\sigma <1$. Recall that $ \phi= iW \psi, \psi= \psi_1+ i \psi_2$. Let $R$ be a fixed but large number so that $ |W| \geq \frac{1}{2}$ for $|y|>R$. Now we define the following norms for (complex) functions $ \psi, h \in L^\infty (\R^2 \backslash \Omega_{\epsilon, x_0}), g \in L^\infty (\partial \Omega_{\epsilon, x_0})$,
\begin{equation}
\| \phi \|_{*}= \| \phi \|_{C^{2, \gamma} ( \R^2 \backslash \Omega_{\epsilon, x_0} \cap \{ |y| <R \}) } +
\end{equation}
\[ +  \Bigg[ \|  \psi_1 \|_{L^\infty ( \R^2 \backslash \Omega_{\epsilon, x_0} \cap \{ |y| >R \} )} +\| <y>^{1+\sigma} \nabla \psi_1 \|_{L^\infty (\R^2 \backslash \Omega_{\epsilon, x_0} \cap \{ |y| >R \}  )} \Bigg]
\]
\[ + \Bigg[ \| <y>^{1+\sigma} \psi_2 \|_{L^\infty (\R^2 \backslash \Omega_{\epsilon, x_0} \cap \{ |y| >R \}  )} +\| <y>^{2+\sigma} \nabla \psi_2 \|_{L^\infty (\R^2 \backslash \Omega_{\epsilon, x_0} \cap \{ |y| >R \}  )} \Bigg],
\]
$$
\| h\|_{**,i}=  \| iW h \|_{C^{0, \gamma} (\R^2 \backslash \Omega_{\epsilon, x_0} \cap \{ |y| >R \} )} + [\| <y>^{2+\sigma} h_1\|_{L^\infty (|y| >R)} + \| <y>^{1+\sigma} h_2 \|_{L^\infty (\R^2 \backslash \Omega_{\epsilon, x_0} \cap \{ |y| >R \} )} ],
$$
$$
\| g\|_{**,b}=   [\| <y>^{2+\sigma} g_1\|_{L^\infty (\partial (\R^2 \backslash \Omega_{\epsilon, x_0}) )} + \| <y>^{1+\sigma} g_2 \|_{L^\infty (\partial (\R^2 \backslash \Omega_{\epsilon, x_0} ) )} ],
$$
\begin{equation}
\| (h, g)\|_{**}= \| h \|_{**,i}+ \| g\|_{**,b}.
\end{equation}

The forms of these norms are motivated by the expressions of (\ref{N1n})-(\ref{N2}). We refer to \cite{DKM, DWZ} for similar definitions.

Under these norms, from (\ref{error3}) and (\ref{bdry}), we can easily derive  the following:
\begin{equation}
\label{error5}
 \| \frac{{\mathbb S} [W] }{iW} \|_{**,i } + \| \frac{\partial W}{\partial \nu} \|_{**,b} \lesssim \epsilon^{\sigma}.
\end{equation}

For the term $\tilde{L}_{\epsilon, 2}$,  we note that $ \epsilon <y> \leq <x> $ and hence we obtain
\begin{equation}
\label{N1m}
\| \tilde{L}_{\epsilon, 2}\|_{**, i}  \lesssim \max |\nabla \Phi_\epsilon (x)| \| \phi\|_{*}  + \max |\nabla \rho_\epsilon|  \| \phi\|_{*}.
\end{equation}

By the estimates of $\rho_\epsilon $ (see (\ref{barrho})) and Theorem \ref{t1}, we see that $\nabla \Phi_\epsilon \to \nabla \Phi^{\delta}, \nabla \rho^\epsilon \to \nabla \rho^\delta = \nabla (1-|\nabla \Phi^\delta |^2)$. Thus it holds
\begin{equation}
|\nabla \Phi_\epsilon|+ |\nabla \rho_\epsilon|  \lesssim  \delta.
\end{equation}
This gives that
\begin{equation}
\label{N1-new}
\| \tilde{L}_{\epsilon, 2}\|_{**, i}  \lesssim  \delta \| \phi\|_{*}.
\end{equation}

For the nonlinear term $ \tilde{N}_\epsilon$, we have
\begin{equation}
\label{N2e}
\| \tilde{N}_\epsilon \|_{**, i} \lesssim \| \phi\|_{*}^{1+\sigma}.
\end{equation}

\subsection{Projected linear and nonlinear problems}

We define
\begin{equation}
\label{Zd}
Z_j:=\frac{ z_j}{1+|y|^4 },  j=0, 1,
\end{equation}
where
$$
z_0= i W,\  z_1= \frac{\partial W}{\partial y_2}.
$$

In this section, our aim is to solve the following {\em projected problem}:
\begin{equation}
\label{Po1}
\left\{\begin{array}{l}
{\mathbb S} [W+ \phi]=\lambda_0 Z_0 + \lambda_1 Z_1 \ \mbox{in} \ \R^2 \backslash \Omega_{\epsilon, x_0},
 \\
 \frac{\partial \phi}{\partial \nu }=-\frac{\partial W}{\partial \nu}  \ \mbox{on} \ \partial \Omega_{\epsilon, x_0},
 \\
Re(\int_{\R^2 \backslash \Omega_{\epsilon, x_0} } \bar{\phi} Z_j)=0, j=0, 1,
\end{array}
\right.
\end{equation}
where $ {\mathbb S} [W+\phi]$ is the solution operator  defined by (\ref{2.00}).

To this end, we first need to consider the following  linear problem
\begin{equation}
\label{linear}
\left\{\begin{array}{l}
\tilde{L}_{\epsilon, 1} [\psi ] + \tilde{L}_{\epsilon, 2} [\psi] = h +\frac{1}{ iW} \lambda_0 Z_0 +\frac{1}{iW} \lambda_1 Z_1 \ \ \mbox{in} \ \R^2 \backslash \Omega_{\epsilon, x_0},\\
\frac{\partial \phi}{\partial \nu }=g \ \mbox{on} \ \partial \Omega_{\epsilon, x_0},
 \\
 \mbox{Re} (\int_{ \R^2\backslash \Omega_{\epsilon, x_0} } \bar{\phi} Z_j)=0, j=0,1.
\end{array}
\right.
\end{equation}

We have the following key  {\em a priori estimates}.
\begin{lemma}
\label{5.1}
There exists a constant $C$, depending on $\gamma, \sigma$ only such that for all $\epsilon$ sufficiently small,  and any solution of (\ref{linear}), it holds
\begin{equation}
\| \phi \|_{*} \leq C (\| h \|_{**, i}+ \| g\|_{**, b}).
\end{equation}
\end{lemma}

\begin{proof} We proceed similar to the proof of Lemma 5.1 in \cite{lw-cpam}. See also similar arguments in \cite{DMK, DWZ}.  The key difference is now that we relax the decay of $\psi_1$ to be bounded only and that we don't have the (odd) symmetry assumption. To overcome this difficulty we use some ideas from \cite{DKW-jdg}.

  We prove it by contradiction. Suppose that there exist a sequence of $\epsilon= \epsilon_n \to 0$, constants $ \lambda_0^n, \lambda_1^n$, and functions $\phi^n, h_n, g_n$ which satisfy (\ref{linear}) with
\begin{equation}
\label{511}
\| \phi^n \|_{*}=1, \| h_n \|_{**,i} =o(1), \| g_n \|_{**, b} =o(1).
\end{equation}

We first note  that $ \lambda_0^n, \lambda_1^n= o(1)$. This follows by multiplying the equation (\ref{linear}) with $(iW) z_0, (iW) z_1$ and integrating by parts. See Lemma 6.2 of \cite{DKW-jdg}.

Next we derive {\em inner estimates} first.  Let $R>0$ be fixed and large.  We claim that $\|\phi^n \|_{L^\infty (B_{4R})} =o(1)$. In fact suppose not. Then by a limiting process we obtain a solution to the linear equation ${\mathbb L}_0$ in $\R^2_{+}$. Thanks to the {\em Key Assumption} (\ref{ker12}), the kernels of ${\mathbb L}_0$ consist of linear combinations of $ z_0$ and $z_1$ only.  But the limit of $\phi^n$ is exactly  orthogonal to the approximate kernels. This is impossible. (This argument is standard now so we omit the details. See \cite{DKM, DWZ, lw-cpam}.)

Next we shall derive {\em outer estimates}: this is the more technical part. (To avoid the clumsy notation we drop the dependence on $n$.) We follow the proof in Section 6 of \cite{DKW-jdg}.
 For $|y|>2R $ the system becomes (see (\ref{L01}))
  \begin{equation}
  \label{51}
 \left\{\begin{array}{l}
   \Delta \psi_1 - 2 \nabla_x \Phi_\epsilon (x_0) \cdot \nabla \psi_2+ {\mathcal O} (<y>^{-2} ) |\nabla \psi|) = f_1, \ \mbox{in} \ \R^2 \backslash \Omega_{\epsilon, x_0} \cap \{ |y|>2R\}, \\
\Delta \psi_2 -2\rho_\epsilon^2 |W|^2  \psi_2 + 2 \nabla_x \Phi_\epsilon (x_0) \cdot \nabla \psi_1  + {\mathcal O} (<y>^{-2})|\nabla \psi|= f_2,  \ \mbox{in} \ \R^2 \backslash \Omega_{\epsilon, x_0},\\
\psi_1= o(1), \psi_2=o(1) \ \mbox{on} \ \R^2 \backslash \Omega_{\epsilon, x_0} \cap \{ |y|=2R\}, \\
\frac{\partial \psi_1}{\partial \nu}= o( <y>^{1+\sigma}), \frac{\partial \psi_2}{\partial \nu}= o( <y>^{2+\sigma}), \ \mbox{on}\ \partial \Omega_{\epsilon, x_0} \cap \{ |y|>2R\}.
  \end{array}
  \right.
  \end{equation}
  where $f_1, f_2$ contain all remaining error terms which are of the order $o( <y>^{2+\sigma}), o(<y>^{1+\sigma})$ respectively.

The first two terms  in the  second equation in (\ref{51}) behavior like $ \Delta -1$, which implies by Maximum Principle
\begin{eqnarray*}
& \ &  \| <y>^{1+\sigma} \psi_2 \|_{L^\infty (|y|>2 R)} +\| <y>^{2+\sigma} \nabla \psi_2 \|_{L^\infty (|y|>2 R)} \\
 & \lesssim &  \|  <y>^{1+\sigma}  \nabla_x \Phi_\epsilon (x_0) \cdot \nabla \psi_1 \| + o(1) \lesssim \max |\nabla_x \Phi_\epsilon (x_0)| \| \psi \|_{*}  + o(1) \lesssim \delta + o(1).
 \end{eqnarray*}
 It remains to consider the first equation in (\ref{51}) only. Similar to  the estimates of Lemma 6.1 of \cite{DKW-jdg}, for linear equation
$$ \left\{\begin{array}{l}
-\Delta \psi = f \ \mbox{in} \ \R^2 \backslash \Omega_{\epsilon, x_0} \cap \{ |y|>2R\}\\
 \psi= o(1) \ \mbox{on} \ \R^2 \backslash \Omega_{\epsilon, x_0} \cap \{ |y|=2R\}, \frac{\partial \psi}{\partial \nu}= g\ \mbox{on}\ \partial \Omega_{\epsilon, x_0} \cap \{ |y|>2R\}
 \end{array}
 \right.$$
it holds that
$$ \| \psi \|_{L^\infty (|y|>2 R)} +\| <y>^{1+\sigma} \nabla \psi \|_{L^\infty (|y|>2 R)}  \lesssim  ( \|  <y>^{2+\sigma} f \|_{L^\infty} +\|  <y>^{1+\sigma} g \|_{L^\infty}  +\| \psi \|_{L^\infty (R <|y| < 3 R)} ). $$

In the first equation in (\ref{51}) we consider the linear term $ 2 \nabla_x \Phi_\epsilon (x_0) \cdot \nabla \psi_2$ as a perturbation, since $ |\nabla_x \Phi_\epsilon | \lesssim \delta$. Therefore we get
$$  \| \psi_1 \|_{L^\infty (|y|>2 R)} +\| <y>^{1+\sigma} \nabla \psi_1 \|_{L^\infty (|y|>2 R)}  \lesssim  ( \|  <y>^{2+\sigma} f_1 \|_{L^\infty} + \| \psi_1 \|_{L^\infty (R_0 <|y| < 3 R)} +\delta \| <y>^{2+\sigma} \nabla \psi_2 \|_{|y|>2R }$$
where the last term can be bounded by $\delta \| \phi\|_{*}$.

All together  we obtain the following outer estimates
$$ \| \psi_1 \|_{L^\infty (|y|>2 R)} +\| <y>^{1+\sigma} \nabla \psi_1 \|_{L^\infty (|y|>2 R)} +  \| <y>^{1+\sigma} \psi_2 \|_{L^\infty (|y|>2 R)} = o(1). $$

Combining both inner and outer estimates, we obtain that $ \|\phi\|_{*}=o(1)$, a contradiction to (\ref{511}).

\end{proof}

\begin{remark}
The condition that $\delta$ is small is only used in the outer estimate argument. We believe that this is just a technical condition.
\end{remark}

We consider now the following projected linear problem
\begin{equation}
\label{linear2}
\left\{\begin{array}{l}
\tilde{L}_{\epsilon, 1} [\psi ] + \tilde{L}_{\epsilon, 2} [\psi]= h +  \sum_{j=0}^1 \lambda_j \frac{Z_j}{i W} \ \ \mbox{in} \ \R^2\backslash \Omega_{\epsilon, x_0},\\
 \mbox{Re} (\int_{ \R^2 \backslash \Omega_{\epsilon, x_0} } \bar{\phi} \frac{\partial W }{\partial y_2})=0, \\
 \frac{\partial \phi }{\partial \nu}=g \ \mbox{on} \ \Omega_{\epsilon, x_0}.
\end{array}
\right.
\end{equation}

We state the following existence result for the projected linear problem.
\begin{proposition}
\label{p4.1}
There exists $\epsilon_0$ such that for all $\epsilon<\epsilon_0$, the following holds: if $ \| (h, g) \|_{**} <+\infty$,  then there exists a  unique solution $\phi= T_\epsilon (h, g)$  to (\ref{linear2}). Furthermore it holds that
\begin{equation}
\| T_\epsilon (h, g) \|_{*} \lesssim \| h \|_{**,i}+\| g\|_{**, b}
\end{equation}
\end{proposition}

\begin{proof}
 The proof is similar to that of [Prop. 4.1,\cite{DKM}]. Instead of solving (\ref{linear2}) in $\R^2\backslash \Omega_{\epsilon, x_0}$, we  solve it in a bounded domain first:
\begin{equation}
\label{linearbd}
\left\{\begin{array}{l}
\tilde{L}_{\epsilon, 1} [\psi ] + \tilde{L}_{\epsilon, 2} [\psi]= h + \sum_{j=0}^1 c_j  \frac{Z_j}{i W} \ \ \mbox{in} \ (\R^2 \backslash \Omega_{\epsilon, x_0}) \cap B_M (0),\\
 \mbox{Re} (\int_{ B_M } \bar{\phi} Z_j )=0, j=0,1,\\
\phi=iW \psi, \frac{\partial \phi}{\partial \nu}=0 \ \mbox{on} \ \partial \Omega_{\epsilon, x_0}, \\
\phi=0 \ \mbox{on} \ \partial B_M (0),
\end{array}
\right.
\end{equation}
where $M >10 R$. By the same proof of a priori estimates, we also obtain the following estimates for any solution $\phi=\phi_M$ of (\ref{linearbd}):
\begin{equation}
\| \phi_M \|_{*} \lesssim \| h \|_{**,i} +\| g\|_{**, b}.
\end{equation}

By working with the Sobolev space $H^1(\R^2 \backslash \Omega_{\epsilon, x_0} \cap B_M (0)) \cap \{ u=0 \ \mbox{on} \  \partial B_M (0) \}$,  the existence then follows from Fredholm alternatives.  (Recall that problem (\ref{4.1}) admits a variational structure
as follows
$$\int \frac{1}{2} \rho_\epsilon^2 |\nabla v + i \Phi_\epsilon \ v|^2 + \frac{1}{4} \rho_\epsilon^4 (1-|v|^4). $$
Here the fact that $ \rho_\epsilon e^{\frac{i \Phi_\epsilon}{\epsilon}}$ satisfies (\ref{2.1}) is used.)

Now letting $M \to +\infty$, we obtain a solution to (\ref{linear2}) with the required properties.

\end{proof}

\subsection{Step 1: projected nonlinear problem}

Finally, we consider the full nonlinear {\em projected problem} (\ref{Po1}). By (\ref{L0}) this is equivalent to
\begin{equation}
\label{nonlinear}
\left\{\begin{array}{l}
\tilde{L}_{\epsilon, 1} [\psi ] + \tilde{L}_{\epsilon, 2} [\psi]+  \tilde{N}_\epsilon [\psi]= \tilde{E} +\lambda_0 \frac{Z_0}{iW} + \lambda_1 \frac{Z_1}{iW} \ \mbox{in} \ \R^2 \backslash \Omega_{\epsilon, x_0},
 \\
 \frac{\partial \phi}{\partial \nu }=-\frac{\partial W}{\partial \nu}  \ \mbox{on} \ \partial \Omega_{\epsilon, x_0},
 \\
Re(\int_{\R^2 \backslash \Omega_{\epsilon, x_0} } \bar{\phi} Z_j)=0, j=0, 1.
\end{array}
\right.
\end{equation}

Using the operator $T_\epsilon$ defined by Proposition (\ref{p4.1}), we can write (\ref{nonlinear}) as
\begin{equation}
\label{psiT}
\psi =  T_\epsilon \circ (-\epsilon \tilde{N}_\epsilon  [\psi]+ \tilde{E} )
\end{equation}
which is equivalent to
\begin{equation}
\psi =  G_\epsilon [\psi]
\end{equation}
where $G_\epsilon$ is the nonlinear operator at the right hand  side of (\ref{psiT}).

Using the error estimate (\ref{error5}) we see that
\begin{equation}
\label{Ee}
 \| \tilde{E} \|_{**, i} \leq C \epsilon^{\sigma}.
\end{equation}

Assuming that
\begin{equation}
\phi \in {\mathbb B}= \{  \| \phi\|_{*} < C \epsilon^{\sigma} \}
\end{equation}
 then we have, using the explicit form of $\tilde{N}_\epsilon [\psi]$ at (\ref{N1n}):
 \begin{equation}
 \label{N1}
 \|\tilde{ N}_\epsilon [\psi]\|_{**, i}= \| <y>^{2+\sigma} \tilde{N}_{\epsilon,1} [\psi]\|_{L^\infty ( |y| >2)} + \| <y>^{1+\sigma} \tilde{N}_{\epsilon,2} [\psi]\|_{L^\infty (|y| >2)}
 \end{equation}
where $ \tilde{N}_\epsilon= \tilde{N}_{\epsilon,1} + i \tilde{N}_{\epsilon,2}$.
Now since $ i \frac{\partial \psi}{\partial x_2}= -\frac{\partial \psi_2}{\partial x_2}+ i \frac{\partial \psi_2}{\partial x_1}$ and
\[ \| i \frac{\partial \psi}{\partial x_2} \|_{**, i} \leq \| <y>^{2+\sigma} \frac{\partial \psi_1}{\partial x_2} \|_{L^\infty (|y|>2)} +\| <y>^{1+\sigma} \frac{\partial \psi_1}{\partial x_2} \|_{L^\infty ( |y| >2)} \leq C \|\psi\|_{*},
\]
 we obtain that
\begin{equation}
\| G_\epsilon [\phi]\|_{*} \leq C (\|\tilde{N}_{\epsilon} [\psi]\|_{**, i} +\|\tilde{E} \|_{**, i}) \leq C \epsilon^{1-\sigma}.
\end{equation}
Similarly, we can also show that
\begin{equation}
\| G_\epsilon [\phi^{'}] -G_{\epsilon} [\phi^{''}]\|_{*} \leq o(1) \| \phi^{'} -\phi^{''} \|_{*}
\end{equation}
for all $\phi^{'}, \phi^{''} \in {\mathbb B}$.

By contraction mapping theorem, we conclude that
\begin{proposition}
\label{p5.1}
There exists a constant $C$, depending on $\gamma, \sigma$ only such that for all $\epsilon$ sufficiently small and $x_0 \in \partial \Omega$ the following holds: there exists a  unique solution $\phi_{\epsilon, x_0}$  to (\ref{nonlinear}) and $ \phi_{\epsilon, x_0}$ satisfies
\begin{equation}
\| \phi_{\epsilon, x_0} \|_{*} \leq C \epsilon^{1-\sigma}.
\end{equation}
Furthermore, $ \phi_{\epsilon, x_0}$ is $C^1$ in $x_0$.

\end{proposition}

\subsection{Step 2: $\lambda_1=0$}

We now solve the reduced problem. From Proposition  \ref{p5.1}, we deduce the existence of  a solution $(\phi, \lambda_0, \lambda_1)= (\phi_{\epsilon, x_0}, \lambda_{0, \epsilon, x_0}, \lambda_{1, \epsilon, x_0} ) $ satisfying
\begin{equation}
\label{nonlinear2}
 {\mathbb S} [W+\phi_{\epsilon, x_0}] = {\mathbb L}_\epsilon [\phi_{\epsilon, x_0}]+ {\mathbb N}_\epsilon  [\phi_{\epsilon, x_0}] +{\mathbb S} [W] = \lambda_0 i Z_0 + \lambda_1 Z_1.
\end{equation}

Multiplying (\ref{nonlinear2}) by $ z_1=\frac{\partial W}{\partial y_2}$ and integrating over $\R^2 \backslash \Omega_{\epsilon, x_0}$, we obtain
\begin{eqnarray*}
 \lambda_1  \int_{\R^2 \backslash \Omega_{\epsilon, x_0}}  | \frac{\partial W}{\partial y_2} |^2= \int_{\R^2 \backslash \Omega_{\epsilon, x_0}}  ( {\mathbb L}_\epsilon  [\phi_{\epsilon, x_0}]+ {\mathbb N}_\epsilon [\phi_{\epsilon, x_0}] +{\mathbb S} [W]) \frac{\partial \bar{W}}{\partial y_2}.
\end{eqnarray*}

We concentrate on the last integral $\int_{\R^2 \backslash \Omega_{\epsilon, x_0}}  {\mathbb S} [W] \frac{\partial \bar{W}}{\partial y_2}$ (which is the dominating term).  From (\ref{SW1}) we have
 \begin{equation}
 \label{SW}
 {\mathbb S} [W]= 2\epsilon  \nabla_x \log \rho_\epsilon \nabla_y W+ 2 i   (\nabla_x  \Phi_\epsilon - |\nabla \Phi^\delta (x_0)|  e_2) \cdot   \nabla_y W+  (\rho_\epsilon^2 - \rho^\delta (x_0)^2) W (1-|W|^2)
 \end{equation}
We project each term in (\ref{SW})  into $ \frac{\partial W}{\partial y_2}$.  For the first term we obtain
$$ Re (2\int_{\R^2\backslash \Omega_{\epsilon, x_0}} \epsilon  \nabla_x \log \rho_\epsilon \nabla_y W \frac{\partial \bar{W}}{\partial y_2} ) = Re (2\epsilon \nabla_x \log \rho_\epsilon (x_0) \int_{\R^2_{+}}\nabla W \frac{\partial \bar{W}}{\partial y_2} )
$$

\begin{equation}
\label{proj1}
= 2\epsilon \frac{\partial  \log \rho^\delta (x_0)}{\partial \tau_{x_0}}  \int_{\R^2_{+}} |\frac{\partial W}{\partial y_2} |^2 + o(\epsilon)
\end{equation}
since the second term in the expansion of $ \rho_\epsilon$ depends on $ d( x, \partial \Omega)$ only (which is in the normal  direction).

The projection of last term  in (\ref{SW}) gives
\begin{equation}
\label{proj2}
Re (\int_{\R^2\backslash \Omega_{\epsilon, x_0}}   (\rho_\epsilon^2 - \rho^\delta (x_0)^2) W (1-|W|^2) \frac{\partial \bar{W}}{\partial y_2}) = \epsilon  \frac{\partial (\rho^\delta)^2}{\partial \tau_{x_0}}  \int_{\R^2_{+}}  y_2  S (1-S^2) \frac{\partial S}{\partial y_2} +o(\epsilon).
\end{equation}

For the projection to the second term in (\ref{SW}) we compute locally, using the estimate (\ref{320}):
\begin{equation}
 \label{proj3}
 2 Re (i \int_{\R^2\backslash \Omega_{\epsilon, x_0}}  (\nabla_x  \Phi_\epsilon - |\nabla \Phi^\delta (x_0)|  e_2) \cdot   \nabla_y W \frac{\partial \bar{W} }{\partial y_2})=   o(\epsilon)
\end{equation}
by symmetry of $W$. For the remaining terms we have
$$\int_{\R^2\backslash \Omega_{\epsilon, x_0}} ( {\mathbb L}_{\epsilon} [\phi_{\epsilon, x_0}] ) \frac{\partial \bar{W}}{\partial y_2}= \int_{\R^2\backslash \Omega_{\epsilon, x_0}} ( {\mathbb L}_\epsilon  [ \frac{\partial \bar{W}}{\partial y_2}]) \phi_{\epsilon, x_0} +o(\epsilon)= o(\epsilon) $$
$$\int_{\R^2\backslash \Omega_{\epsilon, x_0}} ( {\mathbb N}_\epsilon  [\phi_{\epsilon, x_0}] ) \frac{\partial \bar{W}}{\partial y_2} = {\mathcal O} (\epsilon^{2 (1-\sigma)}) = o(\epsilon).$$

Combining (\ref{proj1})-(\ref{proj3})  we obtain
\begin{equation}
\lambda_1 =A_0 \epsilon \frac{\partial }{\partial \tau_{x_0}} (  |\nabla \Phi^\delta |^2) +(\epsilon^{1+\sigma})
\end{equation}
where
$$ A_0=2\int_{\R^2_{+}} |\frac{\partial W}{\partial y_2} |^2 +\int_{\R^2_{+}}  y_2  S (1-S^2) \frac{\partial S}{\partial y_2} >0$$
since each term is strictly positive.

We claim that $ |\nabla \Phi^\delta | (x_0) \not \equiv C$ on $ \partial \Omega$. In fact if so,  since $ \frac{\partial \Phi^\delta }{\partial \nu}=0$, we obtain that $ \frac{\partial \Phi^\delta }{\partial \tau_{x_0} }= C$. Since $\nabla ( (1-|\nabla \Phi^\delta |^2) \Phi^\delta )=0$, by unique continuation  this is impossible. Since $ |\nabla \Phi^\delta | (x_0)$ is not a constant on $\partial \Omega$,  we see that there are at least two points $ x_1, x_2 \in \partial \Omega$ such that $ \frac{\partial}{\partial \tau_{x_1} } |\nabla \Phi^\delta |^2 (x_1) <0 < \frac{\partial}{\partial \tau_{x_2} }  |\nabla \Phi^\delta |^2 (x_2)$. Now we let $x_0$ vary along the segment between $x_1$ and $x_2$, we obtain at least two positions $x_0$ satisfying  $ \lambda_1=0$. We denote this $x_0$ as $ x_{\epsilon}$ and the corresponding solution $v=W+\phi_{\epsilon, x_\epsilon }$ as $v_{\epsilon}$.

\subsection{Step 3: $\lambda_0=0$}

From Step 2,  we have found  a solution $v_\epsilon$ which satisfies
\begin{equation}
\label{6.1}
 \nabla (\rho_\epsilon^2 \nabla v_\epsilon ) + 2i  \rho_\epsilon^2 \nabla_x \Phi_\epsilon \nabla v_\epsilon + \rho_\epsilon^4 v_\epsilon (1-|v_\epsilon|^2)
=\lambda_0 i \rho_\epsilon^2 \frac{W}{1+|y|^4} \ \mbox{in} \ \R^2 \backslash \Omega_{\epsilon, x_{\epsilon}} , \ \frac{\partial v_\epsilon }{\partial \nu }=0\ \mbox{on} \ \partial \Omega_{\epsilon, x_{\epsilon}}.
\end{equation}

Now we multiply (\ref{6.1}) by $\bar{v}_\epsilon $ (the conjugate of $v_\epsilon$)  and integrate by parts, using the fact that
 $ \nabla (\rho_\epsilon^2 \nabla \phi_\epsilon)=0$, we see that
 $$ \int_{\R^2\backslash \Omega_{\epsilon, x_{\epsilon}} }  \rho_\epsilon^2 |\nabla v_\epsilon |^2 + |v_\epsilon |^2 \rho_\epsilon^4 (1-|v_\epsilon|^2) = \lambda_0 i \int_{\R^2 \backslash \Omega_{\epsilon, x_{\epsilon} } } \rho_\epsilon^2 \frac{ W}{1+|y|^4} \bar{v}_\epsilon. $$
 Taking the imaginary part of the above equation, we obtain that
 $$ \lambda_0 Re (\int_{\R^2 \backslash \Omega_{\epsilon, x_{\epsilon} } } \rho_\epsilon^2 \frac{ W}{1+|y|^4 } \bar{v}_\epsilon )=0. $$
 Since $v_\epsilon \sim  W (1+o(1))$, we deduce that $\lambda_0=0$. Theorem \ref{t2} is thus proved.

 In the above computations, we have used the fact that $ v= W e^{i \psi}$ and $ |\nabla \psi_1| = {\mathcal O} (<y>^{-1}), | \psi_2|= {\mathcal O} (<y>^{-1-\sigma})$ so that the boundary integrals vanish at infinity.

\setcounter{equation}{0}
\section{Dirichlet boundary condition}
We discuss in this section how we can adjust the proofs to deal with the Dirichlet boundary conditions.
We consider the Gross-Pitaevskii  equation with Dirichlet boundary condition
\begin{equation}
\epsilon^2 \Delta u + u (1-|u|^2)=0 \ \mbox{in} \ \R^2 \backslash \Omega, \ \  u=0 \ \mbox{on}\ \partial \Omega.
\end{equation}
We first discuss the existence of vortex free solutions. Same as before we let $ u= \rho e^{\frac{\Phi}{\epsilon}}$. Then we have
\begin{equation}
\epsilon^2 \Delta \rho + \rho  (1-|\nabla \Phi |^2-\rho^2)=0 \ \mbox{in} \ \R^2 \backslash \Omega, \ \  \rho=0 \ \mbox{on}\ \partial \Omega
\end{equation}
\begin{equation}
\nabla (\rho^2 \nabla \Phi)=0 \ \mbox{in} \ \R^2 \backslash \Omega.
\end{equation}
For the boundary conditions of $\Phi$ we impose the usual Neumann boundary condition
$$ \frac{\partial \Phi}{\partial \nu}=0\ \ \ \mbox{on} \ \partial \Omega.$$

 The first ansatz is $W_0= (\rho^\delta, \Phi^\delta)$. Similar to the Neumann boundary condition case, we need to add a boundary layer $\rho_1$:
 \begin{equation}
 \epsilon^2 \Delta \rho_1-2 (\rho^\delta)^2 \rho_1= \epsilon^2 \Delta \rho^\delta \ \mbox{in} \ \R^d \backslash \Omega, \ \rho_1= -\rho^\delta \ \ \mbox{on} \ \Omega.
 \end{equation}
 The remaining proofs are similar to the Neumann boundary condition case. We omit the details.

 To construct the second solution, we need to analyze the behavior of the first solution near the boundary and find the corresponding limiting traveling wave equation.

Let us rescale $ x= x_0 +\epsilon y$ where $ x_0\in \partial \Omega$, $ \hat{\Phi}= \frac{ \Phi}{\epsilon}$. (As before we also assume that $ \nu_{x_0}=\vec{e}_1, \tau_{x_0}= \vec{e}_1.$) Then we have
\begin{equation}
\label{n23}
\left\{\begin{array}{l}
\Delta \rho + \rho  (1-|\nabla \hat{\Phi} |^2-\rho^2)=0 \ \mbox{in} \ \R^2_{+}, \ \  \rho=0 \ \mbox{on}\ \partial \R^2_{+}
\\
\nabla (\rho^2 \nabla \hat{\Phi} )=0 \ \mbox{in} \ \R^2_{+}, \ \ \nabla \hat{\Phi} \to b \vec{e}_2
\end{array}
\right.
\end{equation}
where $ b= |\nabla \Phi^\delta (x_0)|$.

There exists a solution to (\ref{n23}) of the following form
\begin{equation}
\rho = \rho_0 (y_1), \  \hat{\Phi} = b y_2
\end{equation}
where $\rho_0$ is the unique solution of the following ordinary differential equation
$$ \rho_0^{''} + \rho_0 ( 1- b^2 -\rho_0^2)=0, 0<y_1<+\infty, \rho_0 (0)=0, \rho_0 (\infty)= \sqrt{1-b^2}$$
Now the limiting  vortex equation becomes
\begin{equation}
\label{newlimit}
\Delta U + 2\rho_0^{'} (y_1) \frac{\partial U}{\partial y_1} + 2i b \frac{\partial U}{\partial y_2} + \rho_0^2 U (1-|U|^2)=0 \ \mbox{in}\ \R^2_{+}, \ \frac{\partial U}{\partial y_1}=0 \ \mbox{on} \ \partial \R^2_{+}
\end{equation}
We claim that for $b$ small we can construct a new solution $U_b$ to (\ref{newlimit}) with two opposing vortices. As in \cite{lw-cpam} we take the initial  ansatz the same as before
$$ W= S_0 ( | x- d e_1|) S_0 (|x+d e_1|) e^{i \theta_{ d e_1} - i \theta_{-d e_1}}$$
The only new error in the equation comes from the interaction with the boundary layer $\rho_0$ which is the following
$$  \frac{\rho_0^{'}}{\rho_0} (\frac{ y_2 y_1 d}{ ( (y_1-d)^2+y_2^2) ( (y_1+d)^2 + y_2^2)}) $$
Note that $ \rho_0^{'} \sim e^{- C y_1}$, near the vortex $y\sim (d, 0)$ it is exponentially small. The $L^1$ norm of this error has the order $O(\frac{1}{d})$.  The rest of the perturbation arguments in \cite{lw-cpam} goes through. We omit the details.

\setcounter{equation}{0}
\section{Proofs of Theorem \ref{t3.1}}
\label{sec6}

In this section, we prove the key nondegeneracy result Theorem \ref{t3.1}. First it is easy to see that the following functions
\begin{equation}
 z_0= iU_c,  \ \ z_1=\frac{\partial U_c}{\partial y_2}
 \end{equation}
 satisfy the equation (\ref{ker10}) and the Neumann boundary condition $ \frac{ \partial \phi}{\partial y_1} (0, y_2)=0$. Hence they  belong to the kernel (\ref{ker1}). To prove the converse statement, we note that $z_2=\frac{\partial U_c}{\partial y_1}$  satisfies the equation (\ref{ker10}), however does not satisfy the Neumann boundary
condition. We now show that this function produces instead a nonzero eigenvalue.  To this end, we go back to the construction process.  The existence of a solution to (\ref{limit3.1}) for $c$ small is proved in the following steps. To align with the proofs in \cite{lw-cpam}, we use the notation $ c=\epsilon$ and assume that $ \epsilon>0$ is small. We first introduce some definitions from \cite{lw-cpam}.

Let $ d \in [\frac{1}{C_1} \frac{1}{ \epsilon },  \frac{C_1}{\epsilon }]$ where $C_1$ is a large constant. We choose the following ansatz
\begin{equation}
V_d (y)= S_0 ( |y- d \vec{e}_1|) S_0 ( |y+d \vec{e}_1|) e^{i \theta_{d \vec{e}_1} -i \theta_{-d \vec{e}_1}}
\end{equation}
where the function $ w^{+} (y)= S_0(|y|) e^{ i\theta}$ is the degree one vortex solution corresponding to the Ginzburg-Landau equation (\ref{limit3.11}).

Clearly be definition $V_d$ satisfies the Neumann boundary condition $ \frac{\partial V_d}{\partial y_1} (0, y_2)=0 $.  We look for solutions of (\ref{limit3.1}) in the form
\begin{equation}
\label{4.0}
 v(y)= \eta_d  ( V_d + i V_d \psi) + (1-\eta_d ) V_d e^{i \psi}
\end{equation}
where $\eta$ is a function such that
\begin{equation}
\label{eta}
 \eta_d = \tilde{\eta} (|z-d \vec{e}_1|)+ \tilde{\eta} (|z+d \vec{e}_1|)
\end{equation}
and $ \tilde{\eta} (s)= 1$ for $ s \leq 1$ and $ \tilde{\eta} (s)=0$ for $s \geq 2$.

We may write $\psi=\psi_1+ i \psi_2$ with $\psi_1, \psi_2$ real-valued. Set
\begin{equation}
v= V_d +\phi, \ \ \ \  \phi= \eta_d  i V_d \psi + (1-\eta_d ) V_d (e^{i \psi} -1).
\end{equation}

We solve (\ref{limit3.1}) in the following two steps:

\medskip

\noindent
{\bf Step 1:}  Fixing $ d \in [\frac{1}{ C_1}\frac{1}{ \epsilon }, \frac{C_1}{\epsilon }]$, we use the reduction method to find a pair $ (c_\epsilon (d), \phi_{\epsilon, d})$ such that
\begin{equation}
\label{Vd}
\left\{\begin{array}{l}
{\mathbb S}_0 [ V_d+\phi_{\epsilon, d}]:= \Delta ( V_d+ \phi_{\epsilon, d}) + i\epsilon \frac{\partial (V_d +\phi_{\epsilon, d})}{\partial y_2} + ( V_d+ \phi_{\epsilon, d}) (1- | V_d+ \phi_{\epsilon, d}|^2)=c_\epsilon \frac{\partial V_d}{\partial d}, \ \mbox{in} \ \R_{+}^2 = \{ y_1>0 \},
\\
\frac{\partial (V_d+\phi_{\epsilon, d})}{\partial y_1} (0, y_2)=0.
\end{array}
\right.
\end{equation}

\medskip

\noindent
{\bf Step 2:} We find a $d= d_\epsilon$ such that $ c_\epsilon (d)=0$.

The expansion of $ c_\epsilon  (d)$ is given by
\begin{equation}
\label{cd}
c_\epsilon (d) = c_1 \epsilon -\frac{c_2}{d} +{\mathcal  O} ( \epsilon^{1+\sigma}).
\end{equation}

\medskip

Let us denote $ u_d= V_d + \phi_{\epsilon, d}$. Similar to arguments in \cite{RW-aihp, W-prsa}, it can be shown that $ u_d$ is $C^1$ in $d$. We denote
$$ \omega= \frac{ \partial u_d}{\partial d} \Big|_{d=d_\epsilon} $$
Note that $\omega$ satisfies the Neumann boundary condition. Formally we differentiate the equation (\ref{Vd}) with respect to $d$ and we obtain
\begin{equation}
{\mathbb S}_0^{'} [ u_d] (\frac{\partial u_d}{\partial d})= \frac{\partial c_\epsilon}{\partial d}  \frac{\partial V_d}{\partial d} + c_\epsilon (d) \frac{\partial }{\partial d} (\frac{\partial V_d}{\partial d}).
\end{equation}
Now let $d=d_\epsilon$ and note that $ c_\epsilon (d_\epsilon)=0$. We see that $\omega$ satisfies
\begin{equation}
\label{ud}
{\mathbb S}_0^{'} [ u_d] (\omega)= \frac{\partial c_\epsilon}{\partial d}  \frac{\partial V_d}{\partial d}
\end{equation}

The equation (\ref{cd}) can be differentiated and it gives
\begin{equation}
\frac{\partial c_\epsilon}{\partial d}= \frac{c_2}{d^2} + {\mathcal O} (\epsilon^{2+\sigma}).
\end{equation}
This argument can be made rigorous, though tedious. We refer interested readers to  similar arguments in \cite{RW-aihp} or \cite{W-prsa}.

Now let $\phi$ be a bounded solution satisfying (\ref{ker1}). Multiplying the equation (\ref{ud}) by $\phi$ and the equation (\ref{ker1}) by $\omega$ we obtain
$$\frac{\partial c_\epsilon}{\partial d}  \int_{\R^2_{+}} \frac{\partial V_d}{\partial d}  \phi=0 $$
which implies that
\begin{equation}
\label{ddV}
\int_{\R^2_{+}} \frac{\partial V_d}{\partial d} \phi=0.
\end{equation}

We now decompose
$$ \phi= \beta_1 \omega + \tilde{\phi} $$
where the following orthogonality conditions are satisfied
 \begin{equation}
 \label{ortho}
 \int_{ | y- d_\epsilon |<1} \tilde{\phi} Z_0= \int_{ | y- d_\epsilon |<1} \tilde{\phi} \omega = \int_{ | y- d_\epsilon |<1} \tilde{\phi} Z_1=0.
\end{equation}

Note that $\tilde{\phi}$ still satisfies the Neumann boundary condition and also the following equation
\begin{equation}
\label{k1}
{\mathbb L}_0 [\tilde{\phi}]= \beta_1 \frac{\partial V_d}{\partial d}.
\end{equation}

We will show that $\beta_1=0$ and then $\tilde{\phi}=0$. From this the statement of Theorem \ref{t3.1} then follows.

 To this end, we first prove that for $|c|$ sufficiently small
\begin{equation}
\label{k2}
\|\tilde{\phi}\|_{L^\infty} \lesssim |\beta_1|.
\end{equation}
This can be proved by a blow-up argument. In fact, suppose this is not true. We find a sequence of solutions to (\ref{k1}), called $ \tilde{\phi}_n$ with $\|\tilde{\phi}_n \|_{L^\infty} \geq n  |\beta_1|, c_n \to 0$. We  divide both sides of (\ref{k1}) by $ \| \tilde{\phi}_n \|_{L^\infty}$ and let $ \hat{\phi}_n (y) = \frac{\tilde{\phi}_n (y+y_{c_n}) }{ \| \tilde{\phi}_n\|_{L^\infty}}$. Letting $n \to +\infty$, we see that the limit $\hat{\phi}_0=\lim_{n \to +\infty} \hat{\phi}_n$ is a bounded solution of the following equation
 $$ \Delta \phi +  (1- |w^{+}|^2)\phi -2 (w^{+} \cdot \phi) w^{+}  =0 \ \ \mbox{in} \ \R^2.$$
 By the nondegeneracy result (\cite{DFK}) we see that $ \hat{\phi}_0= \alpha_0 (i w^{+})+ \alpha_1 \frac{\partial w^{+}}{\partial y_1}+\alpha_2  \frac{\partial w^{+}}{\partial y_2}$ for some constants $\alpha_0, \alpha_1 $ and $\alpha_2$. Now the orthogonality condition  (\ref{ortho}) then  implies that $ \hat{\phi}_0 \equiv 0$.  This yields that $ \hat{\phi}_n \to  0 $ in $C^2_{loc}$.

Since $\omega \sim \frac{\partial V_d}{\partial d}$ we see from (\ref{ddV}) that
$$
\beta_1 \int_{\R^2_{+}} \frac{\partial V_d}{\partial d} \omega = - \int_{\R^2_{+}} \tilde{\phi}\omega = o(\beta_1)
$$
which implies that $ \beta_1 = o(\beta_1)$ and hence $ \beta_1=0$ and $ \tilde{\phi}=0$. So $\phi=0$. This reaches a contradiction.


\bigskip






\begin{thebibliography}{10}

\bibitem{ADP} A. Aftalion, Q. Du and Y. Pomeau,  Dissipative flow and vortex shedding in the
Painleve boundary layer of a Bose-Einstein condensate, {\em Phys. Rev. Lett. } 91(2003), no.9,
090407-1--4.

\bibitem{BS} F. Bethuel, J.-C. Saut,  Travelling waves for the Gross-Pitaevskii equation. I. {\em  Ann. Inst. H. Poincare' Phys. The'or.}  70  (1999),  no. 2, 147-238.

\bibitem{BOS} F. Bethuel, G. Orlandi and D. Smets, Vortex rings for the Gross-Pitaevskii equation,
{\em J. Eur. Math. Soc. } 6(2004), no.1, 17-94.

\bibitem{BGS} F. Bethuel, P.Gravejat and J.-G. Saut, Travelling waves for the Gross-Pitaevskii equation, II, {\em Comm. Math. Phys.} 285(2009), no.2, 567-651.



\bibitem{BBH}  F. Bethuel, H. Brezis and F. He'lein, Asymptotics for the minimization of a Ginzburg-Landau functional, {\em Calc. Var. and PDE.} 1993, 1(3):123-148

\bibitem{BBH1} F. Bethuel, H. Brezis, F. He'lein, Ginzburg-Landau vortices, Boston: Birkha"user, 1994

 \bibitem{bgs} F. Bethuel, P. Gravejat and J.-C.  Saut,  Travelling waves for the Gross-Pitaevskii equation. II. {\em  Comm. Math. Phys.} 285 (2009), no. 2, 567?651.

\bibitem{bers1}   L. Bers, Ezistence and uniqueness of a subsonic pow past a given profile,
{\em Comm. Pure Appl. Math.} 7(1954), 441-504.

\bibitem{bers2} L. Bers, Mathematical aspects of subsonic and transonic gas dynamics,
John Wiley and Sons, New York, 1958.

\bibitem{BW} S. Byun, L. Wang, The conormal derivative problem for elliptic equations with BMO coefficients on Reifenberg flat
domains, {\em Proc. Lond. Math. Soc.} 90 (3) (2005), 245-272.

\bibitem{cds} R. Carles, R. Danchin and J.-C. Saut, Madelung, Gross-Pitaevskii and Korteweg. {\em Nonlinearity} 25 (2012), no. 10, 2843-2873.

\bibitem{CM} D. Chiron and M. Maris,  Rarefaction pulses for the nonlinear Schr?inger
equation in the transonic limit, {\em  Comm. Math. Phys.} 326(2014), 329-392.



\bibitem{C} C. Coste, Nonlinear Schrodinger equation and superfluid hydrodynamics, {\em Eur.
Phys. J. B Condens. Matter Phys. } 1 (1998), 245-253.

\bibitem{DKW-jdg} M. del Pino, M. Kowalczyk and J. Wei,   Entire solutions of the Allen-Cahn equation and complete embedded minimal surfaces of finite total curvature {\em Journal of Differential Geometry} 83(2013), no.1, 67-131.

\bibitem{DKM} M. del Pino, M. Kowalczyk and M. Musso, Variational reduction for Ginzburg-Landau vortices, {\em J. Funct. Anal.} 239(2006), 497-541.


\bibitem{do} G.-C. Dong and B. Ou, Subsonic flows around a body in space. {\em Comm. Partial Differential Equations } 18 (1993), no. 1-2, 355-379.


\bibitem{DFK} M. del Pino, P. Felmer and M. Kowalczyk,
Minimality and nondegeneracy of degree-one Ginzburg-Landau vortex as a Hardy's type inequality.
{\it Int. Math. Res. Not.} (2004), no. 30, 1511--1527.

\bibitem{DWZ} Q. Du, J. Wei and C. Zhao, Vortex solutions of the high-? high-field Ginzburg-Landau model with an applied current. {\em SIAM J. Math. Anal.} 42 (2010), no. 6, 2368-2401.

\bibitem{fg}
R. Finn and D. Gilbarg, Three dimensional subsonicflows and asymptotic
estimates for elliptic partial differential equations, {\em Acta Math.} 98(1957),
265-296.


\bibitem{fpr} T. Frisch, Y. Pomeau and S.  Rica,  Transition to dissipation in a model of superflow. {\em Phys. Rev. Lett.} 69(11)(1992), 1644-1647.


\bibitem{grav1} P. Gravejat, Asymptotics for the travelling waves in the Gross-Pitaevskii equation. {\em  Asymptot. Anal. } 45 (2005), no. 3-4, 227?299.

    \bibitem{grav2} P.  Gravejat, Limit at infinity and nonexistence results for sonic travelling waves in the Gross-Pitaevskii equation. {\em Differential Integral Equations } 17 (2004), no. 11-12, 1213-1232.

        \bibitem{grav3} P.  Gravejat,  Decay for travelling waves in the Gross-Pitaevskii equation. {\em Ann. Inst. H. Poincar?Anal. Non Lin?ire } 21 (2004), no. 5, 591-637.


 \bibitem{grav4} P. Gravejat, A non-existence result for supersonic travelling waves in the Gross-Pitaevskii equation. {\em Comm. Math. Phys.} 243 (2003), no. 1, 93-103.




\bibitem{GR} J. Grant and P. H. Roberts, Motions in a Bose condensate. III. The structure and effective masses of charged and uncharged impurities, {\em J. Phys. A: Math.,
Nucl. Gen.}  7 (1974), 260-279.


\bibitem{huepe1}
C. Huepe and M.E. Brachet, Scaling laws for vortical nucleation solutions in a model of superflow. {\em Phys. D} 140 (2000), no. 1-2, 126-140.

 \bibitem{huepe2} M. Abid, C. Huepe, S. Metens, C. Nore, C.T. Pham, L.S. Tuckerman,  and M. E. Brachet, Gross-Pitaevskii dynamics of Bose-Einstein condensates and superfluid turbulence. {\em Fluid Dynam. Res.} 33 (2003), no. 5-6, 509-544.


\bibitem{JR1} C. A. Jones, S. J. Putterman, and P. H. Roberts, Stability of wave solutions
of nonlinear Schrodinger equations in two and three dimensions, {\em J. Phys A: Math.
Gen. } 19 (1986), 2991-3011.

\bibitem{JR2} C.A. Jones and P.H. Roberts, Motion in a Bose condensate IV, Axisymmetric solitary waves, {\em J. Phys. A} 15(1982), 2599-2619.

\bibitem{jp} C. Josserand and Y. Pomeau,  Nonlinear aspects of the theory of Bose-Einstein condensates. {\em Nonlinearity} 14 (2001), no. 5, R25-R62.

 \bibitem{jpr} C. Josserand, Y. Pomeau and S. Rica,  Vortex shedding in a model of superflow. {\em Phys. D} 134 (1999), no. 1, 111-125.


\bibitem{LL} L. Landau and E. Lifshitz, On the theory of the dispersion of magnetic permeability in ferromagnetic bodies, {\em Phys. Z. Sowj} 8(1935), 153.

\bibitem{lw-cpam} F.-H. Lin and J. Wei,   Traveling Wave Solutions of Schr\"odinger Map Equation {\em Comm. Pure Appl. Math.} 63(2010), no.12, 1585-1621.



 \bibitem{lz-arma} F.-H. Lin and P. Zhang,  Semiclassical limit of the Gross-Pitaevskii equation in an exterior domain. {\em Arch. Ration. Mech. Anal. } 179 (2006), no. 1, 79-107.


\bibitem{LiuWei} Y. Liu and J. Wei, Adler-Moser polynomials and traveling waves solutions of Gross-Pitaevskii, preprint.



\bibitem{LZ} P.I. Lizorkin, Multipliers of Fourier integrals, {\em Proc. Steklov Inst. Math. }89 (1967) 269–290


\bibitem{maris1} M. Maris, Existence of nonstationary bubbles in higher dimensions, {\em J. Math.
Pures Appl.} 81 (2002), 1207-1239.


  \bibitem{maris2} M. Maris, Nonexistence of supersonic traveling waves for nonlinear Schr?inger equations with nonzero conditions at infinity. {\em SIAM J. Math. Anal.}  40 (2008), no. 3, 1076?1103

\bibitem{maris3}  M. Maris, Traveling waves for
nonlinear Schrodinger equations
with nonzero conditions at infinity, {\em Ann. Math.} 178 (2013), 107-182.


\bibitem{pnb} C.-T.  Pham, C. Nore and  M.E. Brachet, Boundary layers and emitted excitations in nonlinear Schr\"{o}inger
superflow past a disk. {\em Phys. D} 210(3-4)(2005), 203-226.

\bibitem{R} S. Rica, A remark on the critical speed of vortex nucleation in the nonlinear
Schrodinger equation, {\em Phys. D} 148 (2001), 221-226.




\bibitem{RW-aihp} O. Rey and J. Wei, Blowing up solutions for an elliptic Neumann problem with sub- or supercritical nonlinearity. Part II: $N \geq 4$, {\em Ann. Non linearie, Annoles de l'Institut H. Poincare} 22(2005), no. 4, 459-484.









\bibitem{shiffman} M. Shiffman: On the ezistence of subsonic flows of a compressible fluid,
{\em Arch. Rational Mech. Anal.} 2(1952), 605-652.

\bibitem{SW} E.M. Stein, G. Weiss, Introduction to Fourier Analysis on Euclidean Spaces, in: Princeton Mathematical Series, vol. 32, Princeton
University Press, Princeton, NJ, 1971.



\bibitem{W-prsa} J. Wei, Uniqueness and critical spectrum of boundary spike solutions, {\em Proc. Royal Soc. Edin. A } 131(2001), 1457-1480.







\end{thebibliography}
\end{document}